\documentclass[11pt]{amsart}
\usepackage[a4paper,margin=3cm]{geometry}
\usepackage{amsmath}
\usepackage{amssymb}

\newtheorem{theorem}{Theorem}[section]
\newtheorem{lemma}{Lemma}
\newtheorem{corollary}{Corollary}
\newtheorem{proposition}{Proposition}
\newtheorem{definition}{\bf Definition}
\newtheorem{remark}{Remark}
\newenvironment{proof of}[1]{{\bf Proof of #1.}}{\hfill$\Box$\\}

\newcommand{\R}{\mathbb{R}}

\newcommand{\ad}{\mathrm{ad}}

\newcommand{\spn}{\mathrm{span}}

\newcommand{\ve}{\varepsilon}
\newcommand{\ba}{\mathbf{a}}
\newcommand{\bb}{\mathbf{b}}
\newcommand{\bc}{\mathbf{c}}

\newcommand{\bx}{\mathbf{x}}
\newcommand{\by}{\mathbf{y}}
\newcommand{\bz}{\mathbf{z}}
\newcommand{\bn}{\mathbf{n}}
\newcommand{\mH}{\mathfrak H}
\newcommand{\nequiv}{\not\equiv}
\newcommand{\mg}{\mathfrak g}

\begin{document}
\title{Canonical frames for distributions of odd rank and corank 2 with maximal first Kronecker index}
\author{Wojciech Kry\'nski
\address{Institute of Mathematics, Polish Academy of Sciences, ul. \'Sniadeckich 8, 00-956 Warszawa,
Poland, E-mail: krynski@impan.gov.pl}
\and
Igor Zelenko
\address{Department of Mathematics, Texas A$\&$M University, College Station, TX 77843-3368, USA;
E-mail: zelenko@math.tamu.edu}} 

\begin{abstract}
We construct canonical frames and find all maximally symmetric models for a natural generic class of  corank 2 distributions
on manifolds of odd dimension greater or equal to $7$. This class of distributions is characterized by the following two
conditions: the pencil of $2$-forms associated with the corresponding Pfaffian system has the maximal possible
first Kronecker index and the Lie square of the subdistribution generated by the kernels of all these $2$-forms is equal to the original distribution. In particular, we show that the unique, up to a local equivalence, maximally symmetric model in this class of distributions with given dimension of the ambient manifold exists
if and only if the dimension of the ambient manifold is equal to $7$, $9$, $11$, $15$ or $8l-3$, $l\in\mathbb N$.
Besides, if the dimension of the ambient manifold is equal to $19$, then there exist two maximally symmetric models, up to a local equivalence, distinguished by certain discrete invariant. For all other dimensions of ambient manifold there are families of maximally symmetric models, depending on continuous parameters.
Our main tool is the so-called symplectification procedure  having its origin in Optimal Control Theory.
 Our results can be seen as an extension of some classical Cartan's results on
 rank 3 distributions in $\mathbb R^5$ to corank 2 distributions of higher odd rank.
 %$(3,5)$-distribution with so called non-integrable square root.
\end{abstract}
\maketitle\markboth{ Wojciech Kry\'nski and Igor Zelenko} {Canonical frames for distributions of odd rank and corank 2}
\section{Introduction}
\indent

%\subsection{Distributions and their symbols}
\subsection{Distributions and their Tanaka symbols}
\emph{A distribution $D$ of rank $l$ on a $n$-dimensional manifold $M$ or an $(l,n)$-distribution} is a subbundle of the tangent bundle $TM$ with $l$-dimensional fiber. \emph{The corank of an $(l,n)$- distribution} by definition is  equal to $n-l$. Obviously corank is equal to the number of independent Pfaffian equations defining $D$.
Distributions  appear naturally in Control Theory as control
systems linear with respect to controls and in
the geometric theory of ordinary and partial differential equations.

The general problem is  to determine equivalence for germs of these geometric
objects with respect to the natural action of the group of germs of diffeomorphisms of $M$.  Except for
several cases such as line distributions, corank one distributions
and rank $(2,4)$-distributions generic
distributions have functional, and, thus, non-trivial differential
invariants.

The basic characteristics of a distribution $D$ is its \emph{weak derived flag} and the \emph{Tanaka symbol}.
% and the \emph{symbol} at a point
By taking iterative brackets of vector fields tangent to a distribution, one obtains the filtration of the tangent bundle.
More precisely, set $D=D^{1}$ and define recursively
%The obvious (but very rough in the most cases) discrete invariant of a
%distribution $D$ at $q$ is so-called \emph{ the small growth vectors} at $q$.
%It is the tuple
%%$\bigl(\dim D(q),\dim D^2(q),\dim D^3(q),\ldots\bigr)$,
%$\{\dim D^j(q)\}_{j\in{\mathbb N}}$, where $D^j$ is the $j$-th power
%of the distribution $D$, i.e.,
$D^{j}=D^{j-1}+[D,D^{j-1}]$, $j>1$. The space  $D^j(q)$ is called the \emph{ $j$th power of the distribution $D$ at a point $q$}. Clearly $D^j\subseteq D^{j+1}$.
The filtration $\{D^j\}_{j\in \mathbb N}$ is called the \emph{weak derived flag} of a distribution
%at a point
and the tuple of dimensions of the subspaces of this filtration at a given point is called the \emph{small growth vector} of the distribution at this point.
A point of $M$ is called a \emph{regular point}  of a distribution if the small growth vector is constant in a neighborhood of this point.

Further, Let $\mg^{-1}(q)\stackrel{\text{def}}{=}D(q)$ and $\mg^{-j}(x)\stackrel{\text{def}}{=}D^{j}(q)/D^{j-1}(q)$ for $j>1$
If a point $q$ is regular, then the graded space $\frak{m}_q=\sum_{\leq -1}
%D^{j+1}(q)/D^j(q)
\mg^j(q)$ can be naturally equipped with a structure of a graded
nilpotent Lie algebra called a \emph{symbol} of the distribution $D$ at a point
$q$.
%This space is endowed naturally with the structure of a graded nilpotent Lie algebra, generated by
%$\mg^{-1}(x)$.
Indeed, let $\mathfrak p_j:D^j(q)\mapsto \mg^{-j}(q)$ be the canonical projection to a factor space. Take $Y_1\in\mg^{-i}(q)$ and $Y_2\in \mg^{-j}(q)$. To define the Lie bracket $[Y_1,Y_2]$ take a local section $\widetilde Y_1$ of the distribution $D^i$ and
a local section $\widetilde Y_2$ of the distribution $D^j$ such that $\mathfrak p_i\bigl(\widetilde Y_1(q)\bigr) =Y_1$
and $\mathfrak p_j\bigl(\widetilde Y_2(q)\bigr)=Y_2$. It is clear that $[\widetilde Y_1,\widetilde Y_2](q)\in D^{i+j}(q)$. Put
\begin{equation}
\label{Liebrackets}
[Y_1,Y_2]\stackrel{\text{def}}{=}\mathfrak p_{i+j}\bigl([\widetilde Y_1,\widetilde Y_2](q)\bigr).
\end{equation}
It is easy to see that the right-hand side  of \eqref{Liebrackets} does not depend on the choice of sections $\widetilde Y_1$ and
$\widetilde Y_2$. Besides, $\mg^{-1}(q)$ generates the whole algebra $\mathfrak{m}(q)$.
A graded Lie algebra satisfying the last property is called \emph{fundamental}.

One can define the \emph{flat distribution $D_{\mathfrak{m}}$ of constant fundamental symbol $\mathfrak{m}$}.
For this let $M(\mathfrak{m})$ be the connected, simply connected Lie group with the
Lie algebra $\mathfrak{m}$ and let $e$ be its identity. Then $D_\mathfrak{m}$ is the left invariant distribution on $M(\mathfrak{m})$ such that $D_{\mathfrak{m}}(e)=\mg^{-1}$.

The notion of symbol is extensively used in works of N.~Tanaka and
his school (\cite{tan1, tan2, mori,mor, yamag}) who developed the prolongation procedure to construct canonical frames (coframes)  for distributions of so-called constant type, i.e. when the symbols at different points are isomorphic as graded Lie algebras.
In particular, as it was proved in \cite{tan1}, for any fundamental symbol $\mathfrak m$  the flat distribution $D_{\mathfrak m}$ has the algebra of infinitesimal symmetries of maximal dimension among all distributions of constant symbol $\mathfrak m$  and this algebra can be described algebraically in terms of the so-called \emph{universal prolongation of the $\mathfrak m$}, which is in essence the maximal (non-degenerate) graded Lie algebra, containing the graded Lie algebra
%$\displaystyle{\bigoplus_{i\leq 0}\mg^i}$
$\mathfrak m$ as its negative part.

%\subsection{Corank2 distributions and pencils or $2-$forms}
 Consider $(2k+1, 2k+3)$-distributions $D$ with small growth vector $(2k+1, 2k+3)$. The case $k=1$, i.e. the case of $(3,5)$-distributions,  was treated already by Elie Cartan in ~\cite{cart10}.  Such distributions have the prescribed symbol
%(see the explicit description of this symbol in \eqref{symb} below for more general situation)
and the flat distribution is nothing but the Cartan distribution
%if the subdistribution $D_2\subset D$ is integrable then it is well known that $D$ is locally equivalent to the so-called Cartan
%distribution
on the space $J^1(\mathbb R,\mathbb R^2)$ of the 1-jets of functions from $\mathbb R$ to $\mathbb R^2$.
%(the explicit normal form for the fla distribution is given in \eqref{symb} below in more general situation).
So it has the infinite dimensional group of symmetries.
Besides, there exists the unique rank 2 subdistribution $\widetilde D\subset D$ such that
$\widetilde D^2\subset D$.
Moreover, the subdistribution $\widetilde D$ is integrable if and only if $D$ is locally equivalent to
the flat distribution.
%is nothing but the Cartan distribution
%if the subdistribution $D_2\subset D$ is integrable then it is well known that $D$ is locally equivalent to the so-called Cartan
%distribution
%on the space $J^1(\mathbb R,\mathbb R^2)$ of the 1-jets of functions from $\mathbb R$ to $\mathbb R^2$. In particular, it has %infinite dimensional algebra of infinitesimal symmetries.
If the sub-distribution $\widetilde D$ is not integrable, then the germ of $\widetilde D$ at some point satisfies $\widetilde D^2=D$ and the small growth vector of $\widetilde D$ is $(2,3,5)$ . So the equivalence problem for $D$ is reduced to the equivalence problem for $\widetilde D$. The subdistribution $\widetilde D$ has constant symbol and the universal Tanaka prolongation of this symbol  is equal to the exceptional Lie algebra $G_2$.

Now consider the case of an arbitrary $k$. Obviously, the Lie algebra structure of the symbol $\mathfrak m(q)=D(q)\oplus T_qM/D(q)$ is encoded   by the map $A_q\in {\rm Hom}\bigl( \bigwedge^2 D(q), T_qM/D(q)\bigr)$ such that $$A_q(X,Y)=[X,Y], \quad X,Y\in D(q),$$
 where the Lie brackets in the right-hand side are as in the symbol $\mathfrak m(q)$.
 Equivalently, one can consider its dual $A_q^*\in {\rm Hom}\bigl((T_qM/D(q)^*,  \bigwedge^2 D(q)^*)$,
\begin{equation}
\label{A*}
A_q^*(p)(X,Y)=p([X,Y])\quad X,Y\in D(q), \,\,p\in  \bigl(T_qM/D(q)\bigr)^*,
\end{equation}
which can be seen as the pencil of skew-symmetric forms on $D(q)$. This pencil is called the \emph{pencil associated with the distribution $D$ at the point $q$}.
%Note that from the condition $D^2=TM$ it follows that $A^*$ is a monomorphism, therefore  he image of turn can be identified with the %pencil of skew-symmetric forms on $D(q)$ as follows:
%to any $p\in \bigl(T_qM/D(q)\bigr)^*$ on can assygn
So, all symbols of such distributions are in one-to-one correspondence with the equivalence classes of pencils of skew-symmetric forms on $(2k+1)$-dimensional linear space. The canonical forms of pencils of matrices are given by the classical theorems of
Weierstrass and
Kronecker (see \cite[Chapter~12]{Gantmacher}). For pencils of skew-symmetric bilinear forms they are specified in
\cite[Section~6]{Gauger}).
With the help of these forms it was shown recently (\cite{radko}) that \emph{for any symbol of  $(2k+1, 2k+3)-$distributions the corresponding flat distribution  has an infinite dimensional algebra of infinitesimal symmetries.}

\subsection{Genericity assumptions and description of main results}
On the other hand, by analogy with the case $k=1$ we can define a natural generic subclass of $(2k+1, 2k+3)-$distributions with finite dimensional algebra of infinitesimal symmetries. For this one can distinguish a special subdistribution $\widetilde D$ of $D$ (may be with singularities), satisfying $\widetilde D^2\subset D$.  The above-mentioned generic subclass of corank 2 distributions will be  defined  according to the weak derived flag  of $\widetilde D$.
%, such that they have finite dimensional algebra of infinitesimal symmetries.
More precisely, let
us fix an auxiliary volume form $\Omega$ on $D(q)$ and for any $p\in \bigl(T_qM/D(q)\bigr)^*$,  define a vector $X_p\in D$ via the relation
\begin{equation}
\label{Xomega}
i_{X_p}\Omega=\wedge^kA_q^*(p),
\end{equation}
%where in the right hand side we have the $k$th exterior power $d\omega^k$ restricted to $D(x)$.
Then the following subspace $\widetilde D(q)$ of $D(q)$
\begin{equation}
\label{tildeD}
\widetilde D(q)={\rm span}\{X_p(q): p\in(T_qM/D(q)^*\}
\end{equation}
is well defined independently of the choice of $\Omega$.
The following statement is immediate from \eqref{Xomega}:

\begin{lemma}
 \label{Xplem}
 The assignment $p\mapsto X_p$ is a  vector-valued degree $k$ homogeneous polynomial on  $(T_qM/D(q)^*$ and $\dim \widetilde D(q)\leq k+1$.
\end{lemma}

It is easy to observe from the definition of $\widetilde D$  that
%the Lie square
\begin{equation}
\label{widesquare}
\widetilde D^2\subset D
\end{equation}
(see also
\cite[Proposition~2]{JKP}). Therefore for a flat distribution the subdistribution $\widetilde D$ is integrable.
%In contrast to the case $k=1$, for $k>1$ the integrability of the subdistribution $\widetilde D$ does not imply the local equivalence %to the flat distribution. Moreover, functional invariants appear here, which are the object for a separate study.

In the present paper we will consider $(2k+1, 2k+3)$-distributions $D$ with $D^2=TM$, satisfying the following two genericity assumptions

\begin{itemize}
\item[(G1)] $\dim \widetilde D\equiv k+1$;

\item[(G2)] $\widetilde D^2=D$.
%The subdistribution $\widetilde D$ is bracket-generating.
\end{itemize}  

Note that under condition (G1) the projectivization of the assignment $p\mapsto X_p$ at any point $q\in M$ defines a \emph{rational normal curve} in the projective space $\mathbb P(\widetilde D(q))$ (or the Veronese embedding of the real projective line $\mathbb {RP}^1$ into
 $\mathbb P(\widetilde D(x))$). In particular, for $k=2$ this curve defines the  quadric  or, equivalently, the  sign-indefinite quadratic form $Q$, up to a multiplication by a nonzero constant on $\widetilde D$.

Condition (G1) can be described in terms of the so-called \emph {first minimal index} or the \emph{ first Kronecker index} of the pencil associated with $D$.
%$\{d\omega|_D(x)\ |\ \omega\in\Gamma(D^\perp)\}$ of 2-forms on the space $D(x)$.
%The canonical forms of pencils of matrices are given by theorems of
%Weierstrass and
%Kronecker (see \cite[Chapter~12]{Gantmacher}). For pencils of skew-symmetric bilinear forms they are specified in
%\cite[Section~6]{Gauger}).
Since $\dim D(q)$ is odd, this pencil is singular, i.e. each form in it has a nontrivial kernel. Moreover, there exists a homogeneous polynomial map
 $B:T_qM/D(q)\rightarrow D(q)$ such that $B_q(p)\in {\rm ker} A_q^*(p)$ and $B_q\neq 0$. The \emph{ first minimal index} or the \emph{first Kronecker index}
 of the pencil associated with distribution at $q$ (and also of the distribution $D$ at $q$) is by definition the minimal possible degree of such polynomial map.

Lemma \ref{Xplem} implies that the first Kronecker index is not greater than $k$. Further, from the Kronecker canonical form for pencils of skew-symmetric matrices \cite[Theorem 6.8]{Gauger} one can get

  \begin{proposition}
  \label{maxequiv}
  The following four conditions are equivalent:
  \begin{enumerate}
  \item The distribution $D$ satisfies condition (G1);
  \item The first Kronecker index of $D$ is equal to $k$ at any point, i.e. it is maximal possible at any point;
  \item For any $q\in M$ and for any $p\in T_qM/D(q)$, $p\neq 0$, the kernel of the corresponding form $A^*(p)$ is one-dimensional or, equivalently, the kernel is spanned by the vector  $X_p(q)$, defined by \eqref{Xomega}.

  \item The distribtuion D has constant symbol isomorphic to the following graded nilpotent Lie algebra $\mg^{-1}\oplus\mg^{-2}$, where
$\mg^{-1}={\rm span}\{\bx_0,\ldots,\bx_k,\by_1,\ldots,\by_k\}$, $\mg^{-2}={\rm span}\{\bz,\bn\}$, and all nonzero products are \begin{equation}
\label{symb}
[\bx_i,\by_{k-i}]=\bz,\quad [\bx_{i+1},\by_{k-i}]=\bn, \quad \forall\, 0\leq i\leq k-1.
\end{equation}
  \end{enumerate}
  \end{proposition}
The item (2) of the previous proposition explains the terminology used in the title of the paper.

Note also that if (G1) holds and $\widetilde D^2$ is strictly contained in $D$ then from item (4) of Proposition \ref{maxequiv} it follows that $\widetilde D^3$ is not contained in $D$ so that in general $D$ is not recovered from $\widetilde D$.
%This motivates our assumption (G2).
Therefore if one wants to study $D$ via  $\widetilde D$ one must to assume (G2).

Can we solve the equivalence problem for the class of distributions, satisfying both (G1) and (G2), in the frame of Tanaka theory, applied for subdistribution $\widetilde D$ for $k>1$? For $k=2$  the subdistribution $\widetilde D$ may have 3 different symbols.
  These symbols can be characterized as follows:
  %it can be shown that in this case $\widetilde D^2=D$.
  The distribution $\widetilde D$ has the distinguished rank 2 subdistribution $\bar D\subset \widetilde D$, satisfying $\bar D^2\subset\widetilde D$. Then, depending on the signature of the restriction of the above mentioned sign-indefinite quadratic form $Q$ to the plane $\bar D(q)$, one has 3 symbols:
 \emph{parabolic, hyperbolic,} or \emph{elliptic}. They are explicitly written in \cite{kuz} (algebras $m7{\_}3{\_}3$ (parabolic case), $m7\_3\_6$ (hyperbolic case), and $m7\_3\_6r$ (elliptic case) in the list there).
 %These symbols can be characterized as follows: it can be shown that in this case $\widetilde D^2=D$. Then the distribution %$\widetilde D$ has the distinguished rank 2 subdistribution $\bar D\subset D$, satisfying $\bar D^2=\widetilde D$. Then there are 3 %cases according to the signature of the restriction of the quadratic form $Q$ to the plane $\bar D(q)$ and they characterize 3 %possible symbol: if $Q|_{\bar D(q)}$ is degenerated, then the symbol of $D$ is isomorphic to $m7_3_3$ from \cite{kuzmich}; %hyperbolic case, when $Q|_{\bar D(q)}$ intersect $C$ is sign-indefinite, corresponds to the symbol $m7_3_6$:
It can be shown
%using the list of $7$-dimensional non-degenerate  fundamental symbols from  \cite {kuz}
that the flat (5,7)-distribution corresponding to the square of the flat distribution with the symbol $m7{\_}3{\_}3$ (i.e. parabolic case)  is the unique, up to the local equivalence, maximally symmetric
among all (5,7)-distributions satisfying conditions (G1) and (G2): the universal prolongation of $m7{\_}3{\_}3$ is $9$-dimensional, while the universal prolongations of $m7\_3\_6$ (hyperbolic case), and $m7\_3\_6r$ (elliptic case) are $8$-dimensional.
The graded Lie algebra symbol $m7{\_3}{\_}3$ is described as follows: $m7{\_3}{\_}3=\mg^{-1}\oplus\mg^{-2}\oplus\mg^{-3}$ where $\mg^{-1}=\spn\{\bx_0,\bx_1,\bx_2\}$, $\mg^{-2}=\spn\{\by_1,\by_2\}$, $\mg^{-3}=\spn\{\bz,\bn\}$ and all nonzero products are
\begin{equation}
\label{sym57}
\begin{split}
~&[\bx_i,\by_{2-i}]=\bz,\qquad [\bx_{i+1},\by_{2-i}]=\bn,\qquad i=0,1;\\
%and
~&[\bx_0,\bx_1]=\by_1,\quad [\bx_0,\bx_2]=\by_2.
\end{split}
\end{equation}

However, starting from $k=3$ the set of symbols of $\widetilde D$ depends on continuous parameters. So in order to apply Tanaka's theory to the considered class of distributions one has to classify all this symbols and to generalize this theory to the distribution with non-constant symbol.

Instead, we use the so-called  \emph {symplectification of the problem} or the \emph{symplectification procedure}. This procedure was already successfully used for other classes of distributions such as rank 2 and rank 3 distributions of so-called maximal class
(\cite{zelvar, zelcart,doubzel1, doubzel2,doubzel3}). It allows to overcome the dependence on symbol in the construction of the canonical frames.

The important object here is the so-called \emph{annihilator $D^\perp$ of $D$}, which is
the subbundle of the cotangent bundle $T^*M$ with the fibers %defined as follows
$
D^\perp(q)=\{p\in T_q^* M\ |\ p(D(q))=0\}.
$
By $\mathbb P T^*M$ denote the projectivization $\mathbb P T^*M$ of the cotangent bundle  $T^*M$, i.e. the fiber bundle over $M$ with the fiber over $q$  equal to the projectivizations of $T^*_qM$. In the same way let $\mathbb P D^\perp$ be the projectivization of $D^\perp$. For a corank 2 distribution $D$ the bundle $\mathbb P D^\perp$ has one-dimensional fibers.
Besides, $\mathbb P D^\perp$ is foliated by the so-called \emph{abnormal extremals} (the characteristic curves of $\mathbb P T^*M$). Thus $\mathbb P T^*M$ is equipped with two rank $1$ distributions $V$ and $C$: $V$ is the distribution tangent to the fibers and $C$ is the distribution tangent to the foliation of abnormal extremals. Besides, the rank 2 distribution $V\oplus C$ is bracket generating. So, the distributions $V$ and $C$  define the so-called \emph{pseudo-product structure} on $\mathbb P D^\perp$. In this way the equivalence problem for the original distribution is reduced to the equivalence problem for such pseudo-product structures.

In the sequel the subdistribution $\widetilde D$ will be denoted by $D_{k+1}$ in order to emphasize its rank.
%The main result of the paper regarding the canonical frames can be formulated as follows (for more precise statement see Theorem %\ref{thm 2} below):
The main results of the paper is the construction of the canonical frame for all $(2k+1,2k+3)$-distribution $D$ with $k>1$ satisfying assumptions (G1) and (G2)  (Theorem \ref{thm 2})  and the description of all maximally symmetric models for $k>2$ (Corollary \ref{cor 6}).
In particular, we show that the dimension of the infinitesimal symmetries of such distributions is not greater than
$2k+6$ if $k\nequiv 1 \mod 4$ and $k>2$, it is not greater than $2k+7$ if $k\equiv 1\mod 4$ and $k>1$, and it is not greater than $9$ if $k=2$ . The latter case $k=2$ also follows from the analysis of the list of 7-dimensional non-degenerate fundamental symbols in \cite{kuz}, but even in this case our construction of the canonical frame is unified for all $(5,7)$-distributions, satisfying conditions (G1) and (G2), independently of the symbol of the corresponding subdistribution $D_3$.  Note that the normal form for the maximally symmetric $(5,7)$-distribution (as the square of the flat dsitribution with the symbol with the product table as in \eqref{sym57})
%which is given via the product table \eqref{sym57} of the symbol of the corresponding subdistribution $D_2$ and %its uniqueness up to a local equivalence
can be obtained from the analysis of our frame as well, but it is too technical to be included here (the case $k=2$ is exceptional as shown in Corollary \ref{cor 2} and it needs a separate analysis, while all $k>2$ can be treated uniformly).

Now let us shortly describe our results from section 4 on the maximally symmetric models in the case $k>2$ . All maximally symmetric models are given as the left invariant distributions on Lie groups corresponding to certain bi-graded nilpotent Lie algebras. \emph{The unique, up to a local diffeomorphism,  maximally symmetric model exist for $k=3$, $k=4$, $k=6$ and $k\equiv 1 \mod 4$}.
Further, if  $k=8$ (i.e. the dimension of the ambient manifold is equal to $19$), then there exist two maximally symmetric models, up to a local equivalence, distinguished by certain discrete invariant. Finally, for $k=7$ and $k>9$ with $k\nequiv 1\mod 4$ there are continuous families of distributions having maximal (i.e. $(2k+6)$-dimensional) algebras of infinitesimal symmetries (for details see Corollary \ref{cor 6} below).

Now let us give an explicit description of the maximally symmetric model for all $k$, when it is unique:

{\bf 1) The case $k=3$.}
A $(7,9)$-distribution satisfying conditions (G1) and (G2) with maximal (i.e 12-dimensional) algebra of infinitesimal symmetries is locally equivalent to the square of the flat distribution with the symbol algebra $\mathfrak m=\mg^{-1}\oplus\mg^{-2}\oplus\mg^{-3}$ where $\mg^{-1}=\spn\{\bx_0,\bx_1,\bx_2,\bx_3\}$, $\mg^{-2}=\spn\{\by_1,\by_2,\by_3\}$, $\mg^{-3}=\spn\{\bz,\bn\}$ and all nonzero products are
\begin{equation}
\label{symm79}
\begin{split}
~&[\bx_i,\by_{3-i}]=(-1)^i\bz,\qquad [\bx_{i+1},\by_{3-i}]=(-1)^{i+1}(i+1)\bn,\quad i=0,1,2;\\
%and
~&[\bx_0,\bx_1]=\by_1,\quad [\bx_0,\bx_2]=\by_2,\quad [\bx_0,\bx_3]=3\by_3,\quad[\bx_1,\bx_2]=-2\by_3.
\end{split}
\end{equation}

{\bf 2) The case $k=4$.}
A $(9,11)$-distribution satisfying conditions (G1) and (G2) with maximal (i.e. 14-dimensional) algebra of infinitesimal symmetries is locally equivalent to the square of the flat distribution with the symbol algebra $\mathfrak m=\mg^{-1}\oplus\mg^{-2}\oplus\mg^{-3}$ where $\mg^{-1}=\spn\{\bx_0,\bx_1,\bx_2,\bx_3,\bx_4\}$, $\mg^{-2}=\spn\{\by_1,\by_2,\by_3,\by_4\}$, $\mg^{-3}=\spn\{\bz,\bn\}$ and all nonzero products are

\begin{equation}
\label{symm911}
\begin{split}
~&[\bx_i,\by_{4-i}]=(-1)^i\bz,\qquad [\bx_{i+1},\by_{4-i}]=(-1)^{i+1}(i+1)\bn,\qquad i=0,1,2,3;\\
~&[\bx_0,\bx_1]=\by_1,\quad [\bx_0,\bx_2]=\by_2,\quad [\bx_0,\bx_3]=-\frac{3}{2}\by_3,\quad [\bx_0,\bx_4]=-4\by_4,\\ ~&[\bx_1,\bx_2]=\frac{5}{2}\by_3,\quad [\bx_1,\bx_3]=\frac{5}{2}\by_4.
\end{split}
\end{equation}

{\bf 3) The case $k=6$.}
A $(13,15)$-distribution satisfying conditions (G1) and (G2) with
maximal (i.e. 18-dimensional) algebra of infinitesimal symmetries is
locally equivalent to the square of the flat distribution with the
symbol algebra $\mathfrak m=\mg^{-1}\oplus\mg^{-2}\oplus\mg^{-3}$
where $\mg^{-1}=\spn\{\bx_0,\bx_1,\bx_2,\bx_3,\bx_4,\bx_5,\bx_6\}$,
$\mg^{-2}=\spn\{\by_1,\by_2,\by_3,\by_4,\by_5,\by_6\}$,
$\mg^{-3}=\spn\{\bz,\bn\}$ and all nonzero products are

\begin{equation}
\label{symm1315}
\begin{split}
~&[\bx_i,\by_{6-i}]=(-1)^i\bz,\qquad
[\bx_{i+1},\by_{6-i}]=(-1)^{i+1}(i+1)\bn,\qquad i=0,1,2,3,4,5;\\
~&
[\bx_0,\bx_1]=-\frac{10}{7}\by_1,\quad
[\bx_0,\bx_2]=-\frac{10}{7}\by_2,\quad
[\bx_0,\bx_3]=-\frac{3}{7}\by_3,\quad\\
~&[\bx_0,\bx_4]=\frac{4}{7}\by_4,\quad
[\bx_0,\bx_5]=\frac{25}{7}\by_5,\quad
[\bx_0,\bx_6]=\frac{60}{7}\by_6,
\\ ~&
[\bx_1,\bx_2]=-\by_3,\quad
[\bx_1,\bx_3]=-\by_4,\quad
[\bx_1,\bx_4]=-3\by_5,\quad
[\bx_1,\bx_5]=-5\by_6,
\\ ~&
[\bx_2,\bx_3]=2\by_5,\quad
[\bx_2,\bx_4]=2\by_6.
\end{split}
\end{equation}

{\bf 4) The case $k\equiv 1 \mod 4$.}
The unique, up to a local equivalence,  maximally symmetric models in the case $k\equiv 1 \mod 4$ can be described using the theory of $\mathfrak{sl}_2(\mathbb R)$ representations. For this let $\mathcal V_k$ be the $(k+1)$-dimensional irreducible $\mathfrak {sl}_2(\mathbb R)$-module, $\mathcal V_k=\mathrm{Sym}^k(\mathbb R^2)$. Recall that the $\mathfrak{sl}_2(\mathbb R)$-module $\mathcal V_k\otimes \mathcal V_l$ with $l\leq k$ decomposes into the irreducible $\mathfrak{sl}_2(\mathbb R)$ submodules as follows:

\begin{equation}
\label{tensor}
\mathcal V_k\otimes \mathcal V_l\cong\bigoplus_{0\leq s\leq l}\mathcal V_{k+l-2s},
\end{equation}
while the $\mathfrak{sl}_2(\mathbb R)$-module $\wedge^2 \mathcal V_k$ decomposes into the irreducible $\mathfrak{sl}_2(\mathbb R)$ submodules as follows:
\begin{equation}
\label{wedge}
\wedge^ {2}\mathcal V_k\cong\bigoplus_{0\leq s\leq \frac{k-1}{2}} \mathcal V_{2k-2-4s},
\end{equation}
(see, for example,\cite{Fulton}).
Let $\sigma_{k,l,s}: \mathcal V_k\otimes \mathcal V_l\rightarrow V_{k+l-2s}$ be the canonical projection w.r.t. the  splitting \eqref{tensor} and
$\tau_{k,s}:\bigwedge ^2\mathcal V_k\rightarrow \mathcal V_{2k-2-4s}$ be the canonical projection w.r.t. the splitting \eqref{wedge}.
Note that the $k$-dimensional subspace appears in the splitting \eqref{wedge} if and only if $k\equiv 1 \mod 4$. In this case it corresponds to the index $s=\frac{k-1}{4}$ in the decomposition in the right-hand side of \eqref{wedge}.

Let $\mathfrak m_k=\mathcal V_k\oplus \mathcal V_{k-1}\oplus \mathcal V_1$. Then in the case $k\equiv 1 \mod 4$ the space $\mathfrak m_k$ can be equipped with the structure of the graded Lie algebra: First, let  $\mg^{-1}=\mathcal V_k$, $\mg^{-2}=\mathcal V_{k-1}$, $\mg^{-3}=\mathcal V_1$. Second,  define the Lie product on $\mathfrak m_k$ by the following two operators:
$$\tau_{k, \frac{k-1}{4}}:\wedge ^2\mathcal V_k\rightarrow \mathcal V_{k-1},\quad \sigma_{k,k-1,k-1}:\mathcal V_k\otimes \mathcal V_{k-1}\rightarrow \mathcal V_1.$$

Let us show that this product satisfies the Jacobi identity i.e. that the map $J:\wedge ^3\mathcal V_k\rightarrow \mathcal V_1$ defined by
$$J(v_1,v_2,v_3)=\sum_{\text{cyclic}}
\sigma_{k,k-1,k-1}\bigl(\tau_{k,\frac{k-1}{4}}(v_1,v_2),v_3\bigr),$$
is identically equal to zero. First note that by constructions the map $J$ is a homomorphism of $\mathfrak{sl}_2(\mathbb R)$-modules, i.e. it commutes with the actions of $\mathfrak{sl}_2(\mathbb R)$ on $\wedge ^3\mathcal V_k$ and $\mathcal V_1$. Assume that $J$ is not identically zero. Then $J$ has to be onto, otherwise its image is a proper $\mathfrak{sl}_2(\mathbb R)$-submodule of $V_1$ which is impossible. Therefore the kernel of $J$ is a $\mathfrak{sl}_2(\mathbb R)$-submodule of $\wedge ^3V_k$ of codimension $2$. On the other hand, $\wedge^3\mathcal V_k$  does not contain such submodule, because  the $2$-dimensional module $\mathcal V_1$ does not appear in the decomposition of $\wedge ^3\mathcal V_k$ into the irreducible $\mathfrak{sl}_2(\mathbb R)$-submodules.
To prove this recall that the number of appearances of the module $\mathcal V_l$ in this decomposition is equal to
$N_k(l)-N_k(l+2)$, where
$$N_k(l)=\#\left\{(i_1,i_2,i_3)\in \mathbb Z_{{\rm odd}}^3: -k\leq i_1<i_2<i_3\leq k,\, i_1+i_2+i_3=l\right\}$$
and $\mathbb Z_{{\rm odd}}$ denotes the set of odd integers. In other words, $N(l)$ is the number of non-ordered triples of pairwise distinct odd integers between
$-k$ and $k$ with the sum equal to $l$.
%From the parity arguments $N_k(l)$ is equal to zero for even $l$, which implies
%that $\mathcal V_0$ (and more generally $\mathcal V_l$ with even $l$) does not appear in this decomposition.
%because  only even dimensional submodules appear there by the parity arguments.
The module $\mathcal V_1$ does not appear in this decomposition for $k=2s+1$, $s\in\mathbb N$, because in this case $N_k(1)=N_k(3)$(=
%for  any $k=2l+1$, $l\in \mathbb N$
%$$N(l)=\#\left\{(i_1,i_2,i_3)\in \mathbb Z_{{\rm odd}}^3: \begin{array}{c}-k\leq i_1<i_2<i_3\leq k,\\ %i_1+i_2+i_3=1\end{array}\right\}=
%  \#\left\{(i_1,i_2,i_3)\in \mathbb Z_{{\rm odd}}^3: \begin{array}{c}-k\leq i_1<i_2<i_3\leq k,\\ i_1+i_2+i_3=3\end{array}\right\},$$
%the number of non-ordered triples of pairwise distinct odd integers between
%$-k$ and $k$ with the sum equal to $1$ is the same as the number of non-ordered triples of pairwise distinct odd numbers between
%$-k$ and $k$ with the sum equal to $3$
%where $\mathbb Z_{{\rm odd}}$ denotes the set of odd integers
$\frac{s(s+1)}{2}$).
As a matter of fact, we have proved more general fact that any homomorphism of $\mathfrak{sl}_2(\mathbb R)$-modules $\wedge ^3\mathcal V_k$ and $\mathcal V_1$ is identically equal to zero.

%A $(2k+1,2k+3)$-distribution satisfying conditions (G1) and (G2) with maximal (i.e $2k+7$-dimensional) algebra of infinitesimal symmetries is locally equivalent to
Further, the square of the flat distribution $\mathfrak D_k$ with the symbol algebra $\mathfrak m_k$ has $(2k+7)$-dimensional algebra of infinitesimal symmetries isomorphic to the natural semi-direct sum of $\mathfrak{gl}_2(\mathbb R)$ and $\mathfrak m_k$.
Indeed, the algebra $\mg^0$ of all derivations of the symbol $\mathfrak m_k$ preserving the grading contains the image of the irreducible embedding of $\mathfrak {sl}_2(\mathbb R)$ into $\mathfrak{gl}(\mathcal V_k)$ and the grading element. Therefore $\mg^0$ is at least $4$-dimensional and by \cite{tan1} the algebra of infinitesimal symmetries of the distribution $\mathfrak D_k$ (and also of its square) is at least $(2k+7)$-dimensional. On the other hand, by Theorem \ref{thm 2} below this algebra is at most $(2k+7)$-dimensional. By Corollary \ref{cor 6} the distribution  $\mathfrak D_k$ is the unique, up to the local equivalence, maximally symmetric model of distributions from the considered class
for $k\equiv 1 \mod 4$.
%MFrom Corollary \ref{cor 6} below a $(2k+1,2k+3)$-distribution, $k\equiv 1 \mod 4$, satisfying conditions (G1) and (G2) with maximal (i.e $2k+7$-dimensional) algebra of infinitesimal symmetries is locally equivalent to $\mathfrak D_k$.

%\begin{theorem}\label{thm 1}
%Assume that $(2k+1,2k+3)$-distribution $D$ has maximal first Kronecker index and $D_{k+1}$ is completely non-holonomic.
%\begin{enumerate}
%\item If $k\equiv 1\mod 4$ and $k\geq 5$, then there exists a canonical frame on rank 3 bundle over an open and dense subset of $\mathbb P(D^\perp)$.
%\item If $k\not\equiv 1\mod 4$ and $k\neq 2$, then there exists a canonical frame on rank 2 bundle over an open and dense subset of $\mathbb P(D^\perp)$.
%\item If $k=2$ then there exists a canonical frame on an open and dense subset of $D^\perp$.
%\end{enumerate}
%\end{theorem}

%For any $k>1$ we also describe all distributions having maximal algebra of infinitesimal symmetries (see Corollary \ref{cor 6}).
%For example, ???
\section{Symplectification procedure}
\setcounter{equation}{0}
\setcounter{theorem}{0}
\setcounter{lemma}{0}
\setcounter{proposition}{0}
\setcounter{corollary}{0}
\setcounter{remark}{0}

\subsection{ Characteristic rank 1 distribution on $\mathbb P (D^\perp)$}
Let us describe the process of symplectification of the problem.
For this first let us recall some standard notions from Symplectic Geometry.
Let $\tilde\pi:T^*M\mapsto M$ be the canonical projection. For any
$\lambda\in T^*M$, $\lambda=(p,q)$, $q\in M$, $p\in T_q^*M$, let
$\varsigma(\lambda)(\cdot)=p(\tilde\pi_*\cdot)$ be the canonical Liouville
form and $\sigma=d\varsigma$ be the standard symplectic structure on
$T^*M$. Given a function $H:T^*M\mapsto \mathbb R$ denote by $\vec H$
the corresponding Hamiltonian vector field defined by the
relation $i_{\vec H}\sigma
%(\vec G,\cdot)
=-d\,H$. Given a vector field $X$ on $M$ define the function $H_X:T^*M\rightarrow \mathbb R$, the \emph{quasi-impulse} of $X$, by
$H_X(\lambda)=p\bigl(X(q)\bigr)$, where $\lambda=(p,q)$, $q\in M$, $p\in T_q^*M$. The corresponding Hamiltonian vector field $\vec H_X$ on $T^*M$ is called the \emph{Hamiltonian lift of the vector field $X$}. It is easy to show that $\tilde\pi_*\vec H_X=X$.

%Further if $\mathcal B$ is a smooth vector bundle over $M$ then the sheaf of all smooth sections of $\mathcal B$ is denoted %$\Gamma(\mathcal B)$.
As before, let $D$ be a $(2k+1, 2k+3)$-distribution  with $D^2=TM$. If $\mathcal B$ is a smooth vector bundle over $M$, then the sheaf of all smooth sections of $\mathcal B$ is denoted by $\Gamma(\mathcal B)$.
For any vector field  $Y\in \Gamma(D)$ and any $\lambda=(p,q)\in D^\perp$, where  $q\in M$, $p\in T_q^*M$, the vector
$\vec H_Y (\lambda)$ depends on the vector $Y(q)$ only. This implies that
that  for any $\lambda\in D^\perp$ we set
$$\vec H_D(\lambda)={\rm span}\{\vec H_Y(\lambda): Y \in \Gamma (D)\},$$
then the map $\tilde\pi_*|_{\vec H_D(\lambda)}\colon \vec H_D(\lambda)\rightarrow D\bigl(\tilde\pi(\lambda)\bigr)$ is an isomorphism. The space
$\vec H_D(\lambda)$ is called the \emph
{Hamiltonian lift of the distribution $D$} at $\lambda\in D^\perp$.

Further, since $D^\perp$ is an odd dimensional manifold, the restriction $\sigma(\lambda)|_{D^\perp}$ of the standard symplectic form $\sigma$ on $D^\perp$ has a nontrivial kernel for any $\lambda\in D^\perp$.  This kernel can be described in term of the the $\vec H_D(\lambda)$. Note that the space $(T_qM/D(q))^*$ is identified canonically with the space $(D^\perp)(q)$. Therefore the map $A_q^*$ from \eqref{A*} can be seen as an element of ${\rm Hom}\bigl((D^\perp)(q),  \bigwedge^2 D(q)^*)$. Then it is not hard to show that
for all  $\lambda=(p,q)\in  D^\perp$ one has
\begin{equation}
\label{ker}
{\rm ker} \,\sigma(\lambda)|_{D^\perp}=\vec H_D(\lambda)\cap T_\lambda (D^\perp)=\{v\in \vec H_D(\lambda): \tilde\pi_*v\in {\rm ker} \, A_q^*(p)\}.
%, \quad \forall \lambda=(p,q)\in  D^\perp.
\end{equation}
%One can easily argue that under the assumption $\rk\,D_{k+1}=k+1$
Hence from items (2) and  (3) of Proposition \ref{maxequiv} it follows that for corank 2 distributions with maximal first Kronecker index  ${\rm ker} \,\sigma(\lambda)|_{D^\perp}$
is one dimensional at any point $p\in D^\perp_0$, where $D^\perp_0$ denotes the annihilator of $D$ without the zero section.
%excluded.
In other words, $\ker\sigma|_{D^\perp}$ defines  a rank 1 distribution $\widetilde C$ on $D^\perp_0$.
%The integral curves of this distribution define \emph{a characteristic foliation of $D$}.
Besides, from \eqref{ker} it follows that

%\begin{equation}

\begin{equation}
\label{ker1}
\widetilde\pi_* \widetilde C(\lambda)
%{\rm ker} \,\sigma(\lambda)|_{D^\perp}
=\{\mathbb R X_p(q)\}, \quad D_{k+1}(q)={\rm span}
\{
%{\rm ker} \,\sigma(\lambda)|_{D^\perp}
\widetilde\pi_*\bigl(\widetilde C(\lambda)\bigr): \lambda\in (D^\perp)_0(q)\}.
\end{equation}

% Moreover, it can be shown that the vectors tangent to the singular curves span exactly $D_{k+1}$ (see \cite{JKP} for details).

% and he leaves of this foliation are called \emph{abnormal extremals} .

 Let, as before,  $\mathbb P(D^\perp)$ be the projectivization of the annihilator. Since $\sigma$ is preserved by the flow of the Euler vector field on $D^\perp$ , the rank 1 distribution $\widetilde C$ on $D^\perp_0$ is well projected to the rank 1 distribution $C$ on $\mathbb P(D^\perp)$. This rank 1 distribution is called the \emph{characteristic distribution associated with $D$}.
%The distribution of tangent spaces to characteristic foliation on $P(D^\perp)$ will be denoted $C$. $C$ has rank 1.
Integral curves of $C$ are called \emph{characteristic curves} or \emph {abnormal extremals}  of $D$.
%The projections of characteristic curves to the base manifold $M$ are called \emph{singular curves} of $D$.
The reason for the latter name is that by the Pontryagin Maximum Principle in any variational problem on $M$ with non-holonomic constraints defined by $D$  these curves are exactly the extremals with zero Lagrange multiplier near the functional.
 %in any variational problem on $M$ with non-holonomic constraints defined by $D$.

\subsection{Canonical filtrations on $T\mathbb P(D^\perp)$}
Now let $\pi\colon \mathbb P(D^\perp)\to M$ be the canonical projection. Define the lifts of $D$ and $D_{k+1}$  to $\mathbb P(D^\perp)$ by the formulae
$$
H=\pi^{-1}_*(D),\qquad H_{k+1}=\pi^{-1}_*(D_{k+1}).
$$
Note that by constructions the characteristic distribution $C$ is contained in $H_{k+1}$ (see \eqref{ker1}). Set
$$
V=\pi^{-1}_*(0),
$$
 In other words, $V$ is the vertical distribution on $\mathbb P(D^\perp)$. Note that $V$ has rank $1$, because the fibres of $\mathbb P(D^\perp)$ are homeomorphic to a circle. By definition of $V$  and relation \eqref{widesquare} we have
%\begin{proposition}\label{prop 2}
\begin{equation}
\label{VHC}
[V,H_{k+1}]\subset H_{k+1},\qquad  [V,H]\subset H,\qquad [C,H_{k+1}]\subset H.
\end{equation}
%\end{proposition}

An important observation is that for each $\lambda\in\mathbb P(D^\perp)$  the spaces $H_{k+1}(\lambda)$ and $H(\lambda)/H_{k+1}(\lambda)$ are equipped with the natural filtrations. The filtration on $H_{k+1}$ is described  by the following recursive formula
%Recall that from Lemma \ref{lemma 0} it follows that $H_{k+1}$ has its own filtration:
%$$
%L_0\subset L_1\subset\cdots\subset L_{k-1}\subset L_k=H_{k+1},
%$$
%where
\begin{equation}
\label{filt1}
L_0(\lambda)=V(\lambda)\oplus C(\lambda), \quad L_{i+1}(\lambda)=L_i(\lambda)+
[V,L_i](\lambda).
\end{equation}

Given a vector $v$ in a linear space denote by $[v]$ its equivalence class in the corresponding projective space.
Using \eqref{ker1} one gets easily that for any $\lambda=([\bar p],\bar q)\in \mathbb P (D^\perp)$ the projectivization of the space $\pi_* L_i$ coincides with the $i$-th osculating subspace at the point $[\bar p]$ to the curve $[p]\mapsto [X_p]$, $p\in D^\perp_0(\bar q)$  in $\mathbb P(D_{k+1}(q))$. Since the latter is a rational normal curve in $\mathbb P(D_{k+1}(q))$, we obtain the following filtration of $H_{k+1}(\lambda)$
\begin{equation}
\label{filtl}
L_0\subset L_1\subset\cdots\subset L_{k-1}\subset L_k=H_{k+1},
\end{equation}
where $\dim L_i(\lambda)=i+2$.

Now let us describe the natural filtration on the spaces $H(\lambda)/H_{k+1}(\lambda)$.
Recall that there is the canonical quasi-contact distribution $\Lambda$ on $\mathbb P(D^\perp)$ induced by the Liouville form $\varsigma$ on $T^*M$  %Denote it by $\Lambda$, i.e.
as follows
$$
\Lambda=\rm {pr}_*(\ker\varsigma)\subset T\mathbb P(D^\perp),
$$
where $\rm{pr}\colon D^\perp_0\to P(D^\perp)$ is the quotient mapping. Since $d\varsigma=\sigma$ and $C={\rm ker}\,\sigma|_{D^\perp}$, the distribution $C$
is the Cauchy characteristic of $\Lambda$, i.e. $[C,\Lambda]\subset \Lambda$ and $C$ is the maximal subdistribution with this property. Since by constructions $H\subset\Lambda$, it implies that
%\begin{proposition}\label{prop 3}
%$$
%\begin{equation}
%\label{CHL}
$[C,H]\subset \Lambda$.
%\end{equation}
%$$
%\end{proposition}

%The spaces $H_{k+1}/V$, $H/H_{k+1}$ and $\Lambda/H$ are well defined quotient vector bundles over $P(D^\perp)$.
If $h\in\Gamma(C)$ is a locally non-vanishing section of the characteristic distribution $C$ then by \eqref{VHC} the Lie brackets
%$ad h(\cdot)=
$[h,\cdot]$ at $\lambda$ define the following morphism
%of vector bundles
\begin{equation}
\label{morph}
%\ad_h\colon H_{k+1}(\lambda)/V(\lambda)\to H(\lambda)/H_{k+1}(\lambda),\qquad
\ad_h\colon H(\lambda)/H_{k+1}(\lambda)\to\Lambda/H(\lambda).
\end{equation}
%It follows from the assumption of maximal Kronecker index that
First note that this map
%$\ad_h\colon H/H_{k+1}\to\Lambda/H$
is onto. Otherwise, $\pi_* \bigl(C(\lambda)\bigr)$ is the common kernel for all forms $A^*(p)$, $p\in D^\perp\bigl(\pi(\lambda)\bigr)$,  of the pencil associated with $D$ at $\pi(\lambda)$, which contradicts the assumption of maximality of the first Kronecker index.
Note that %$\rk\,H_{k+1}/H=k+1$,
${\rm rank}\,H/H_{k+1}=k$ and ${\rm rank}\,\Lambda/H=1$. Therefore the kernel of $\ad_h\colon H/H_{k+1}\to\Lambda/H$ has rank $k-1$ and it defines a corank one subdistribution
$
K\subset H.
$
Note that the morphism in \eqref{morph} is multiplied by a nonzero constant, if one chooses another $h\in\Gamma(C)$. Therefore
the distribution $K$ does not depend on this choice.

We have a similar picture on the base manifold $M$. For any $p\in D^\perp(q)$ consider the morphism
${\rm ad}_{X_p}:D(q)\rightarrow {\rm ker}\, p$. Then the codimension of ${\rm ker}\, {\rm ad}_{X_p}$ in $D(q)$ is equal to $1$ and
%for $\lambda=(p,q)$ we have
$$\pi_* K(\lambda)={\rm ker}\, {\rm ad}_{X_p}, \quad \lambda=(p,q)\in \mathbb P (D^\perp).$$
Further, let $Y_p=({\rm ker}\, {\rm ad}_{X_p})/D_{k+1}(p)\subset D(q)/D_{k+1}(q)$ and
$$Z_p=\{\varphi\in (D(q)/D_{k+1}(q))^*: \varphi(Y_p)=0\}.$$
Note that $\dim Z_p=1$.
Then, using the normal form of the symbol from the item (4) of Proposition \ref{maxequiv}, it is not hard to get that the assignment $[p]\mapsto [Z_p]$  defines a \emph{rational normal curve} in $\mathbb P\bigl((D(q)/D_{k+1}(q))^*\bigr)$.

Now let us construct the filtration on $K$ inductively using kernels of natural mappings generated by the iterative brackets with $V$.
Namely, set $K_k=H$, $K_{k-1}=K$, and assume by induction that

\begin{equation}
\label{contr1} K_i(\lambda)=
%{\mathcal D}^{(i-1)}\Lambda (\tau)+
\left\{x\in  K_{i+1}(\lambda):
\begin{array}{l}
\exists\,X\in\Gamma(K_{i+1})\,\,
%\text {a vector field}\,\,
%{\mathcal V}\subset {\mathcal J}_{(i-1)}\,\,
\text {with}\,\, X(\lambda)=x\\
%{\mathcal V}(\lambda)=v
\text {such that}\,\,
%\bigr[{\mathcal C}, {\mathcal
%V}\bigl](\lambda)\in {\mathcal J}_{(i-1)}(\lambda),
[V, X](\lambda)\in
K_{i+1}(\lambda)
\end{array}\right\}, \quad i<k-1
\end{equation}
By constructions $H_{k+1}(\lambda)\subset K_i(\lambda)$ for any $i$.  Set $F_i(\lambda)=K_i(\lambda)/H_{k+1}(\lambda)$. It turns out that the $\pi_*K_i(\lambda)/D_{k+1}(q)$ can be described in terms of the $(k-1-i)$-th osculating space of the curve $[p]\mapsto [Z_p]$. Namely, $\pi_*K_i(\lambda)/D_{k+1}(q)$ is exactly the space of all vectors, annihilated by all elements of these osculating space
(recall that the latter space belong to $\mathbb P\bigl((D(q)/D_{k+1}(q))^*\bigr)$). Since the curve  $[p]\mapsto [Z_p]$ is the rational normal curve, we get that the flag
$\{F_i(\lambda)\}_{i=1}^{k-1}$ is complete, i.e.

\begin{equation}
\label{filtK}
0\subset F_1(\lambda)\subset\cdots\subset F_{k-1}(\lambda)\subset F_k =H(\lambda)/H_{k+1}(\lambda),\quad \dim F_i(\lambda)=i.
\end{equation}
To summarize, filtrations \eqref{filtl} and \eqref{filtK} are obtained with the help of osculating subspaces
to two rational normal curves: $[p]\rightarrow [X_p]$ in $\mathbb P D_{k+1}$ and $[p]\rightarrow [Z_p]$ in $\mathbb P\bigl((D(q)/D_{k+1}(q))^*\bigr)$.
%\begin{lemma}
%\label{filt2lem}
%$\dim L_i(\lambda)=i+2$
%\end{lemma}

%take any non-zero section $e_0\in\Gamma(V)$ and consider the action
%$$
%\ad_{e_0}\colon K\to H/K.
%$$
%This map is onto: if not then $[V,K]=K$ and $K$ projects to a rank $2k$ subdistribution $D_{2k}\subset D$ such that $[D_{k+1},D_{2k}]\subset D$ and this violate the assumtion that Kronecker index is maximal (compare Proof of Corollary \ref{cor 2}).
%Let $K_{k-2}$ be the kernel of $\ad_{e_0}\colon K\to H/K$ and consider
%$$
%\ad_{e_0}\colon K_{k-2}\to K/K_{k-2}.
%$$
%As before this map is onto and thus we have well defined kernel $K_{k-3}$ of $\ad_{e_0}\colon K_{k-2}\to K/K_{k-1}$. We can proceed further and by induction define:
%$$
%K_0\subset K_1\subset\cdots\subset K_{k-2}\subset K\subset H.
%$$
%Note that
%$$
%K_0=H_{k+1}.
%$$
%If we pass to the quotient space we get the following filtration of $H/H_{k+1}$:
%$$
%\{0\}\subset F_1\subset\cdots\subset F_{k-1}\subset F_k=H/H_{k+1},
%$$
%where $F_i=K_i/H_{k+1}$ for $i=1,\ldots,k-2$ and $F_{k-1}=K/H_{k+1}$. Note that $F_i$ has rank $i$.
%
%%Recall that from Lemma \ref{lemma 0} it follows that $H_{k+1}$ has its own filtration:
%%$$
%%L_0\subset L_1\subset\cdots\subset L_{k-1}\subset L_k=H_{k+1},
%%$$
%%where
%%$$
%%L_0=V\oplus C
%%$$
%%and $L_{i+1}=[V,L_i]$. In the quotient space we get:
%%$$
%%\{0\}\subset E_0\subset E_1\subset\cdots\subset E_{k-1}\subset E_k=H_{k+1}/V,
%%$$
%%where $E_i=L_i/V$ has rank $i+1$.

Till now we used assumption (G1) but not (G2). Now we will assume the following condition weaker than $(G_2)$: the distribution  $D_{k+1}$ is not integrable. From this we can extract an additional information from filtrations \eqref{filtl} and \eqref{filtK} in the form of certain integer-valued invariants, which will be important in the sequel. First let
\begin{equation}
\label{Ar}
\mathcal A_r(\lambda)=H_{k+1}(\lambda)+{\rm span}\{[L_s,L_t](\lambda): s+t\leq r, 0\leq s,t\leq k\}
\end{equation}
Obviously, $\mathcal A_r(\lambda)\subseteq\mathcal A_{r+1}(\lambda)$.
Since $D_{k+1}$ is not integrable there exists an integer $r$, $1\leq r\leq 2k-1$ such that
%$[L_s,L_t](\lambda)
$$\mathcal A_r(\lambda)\neq H_{k+1}(\lambda).$$ Let
%$$
%w(\lambda)=\min\{s+t\ |\ [L_s,L_t](\lambda)\neq H_{k+1}(\lambda)\}
%$$
$$w(\lambda)= \min\{r\ |\ \mathcal A_r(\lambda)\neq H_{k+1}(\lambda)\}, $$
and
$$
i(\lambda)=\min\{i\ |\ \mathcal A_{w(\lambda)}\subset K_i(\lambda)\}.
%\ A_{w(\lambda)}\nsubseteq K_{i-1}(\lambda)\}
%\  s+t=w(\lambda)\}.
$$

Given $q\in M$ let
$$
w_D(q)=\min\{w(\lambda)\ |\ \lambda\in \mathbb P(D^\perp)(q)\}
$$
and
$$
i_D(q)=\max\{i(\lambda)\ |\ \lambda\in \mathbb P(D^\perp)(q)\}
$$
The numbers $w_D(q)$ and $i_D(q)$ are integer-valued invariants of the distribution $D$ at $q$.
A point $q\in M$
is said to be \emph{regular} if $w_D$ and $i_D$ are constant in a neighborhood of $q$.
By constructions, the function $w(\lambda)$ is upper semicontinuous and the function $i(\lambda)$ is lower semicontinuous.
It implies that
the set of regular points is open and dense subset of $M$.
%From now on we assume that $w_D(q)\equiv w_D$ and $i_D(q)\equiv i_D$ for some integers $w_D$ and $i_D$
%%There exists an open and dense
Also let  $\mathcal R_1 =\{\lambda\in\mathbb P(D^\perp): w(\lambda)= w_D\bigl(\pi(\lambda)\bigr),
i(\lambda)= i_D\bigl(\pi(\lambda)\bigr)\}$.
%Then $\mathcal R_1$ is open and dense subset of $\mathbb P(D^\perp)$. Moreover,
Then the  intersection of $\mathcal R_1$ with any fiber of $\mathbb P(D^\perp)$ is open set in the Zariski topology of this fiber.
%on $\mathbb P(D^\perp)_{reg}$.
%The subset $\mathbb P(D^\perp)_{reg}$ is called \emph{regular} part of $\mathbb P(D^\perp)$.

We list several properties of the numbers $w_D$ and $i_D$.

\begin{lemma}\label{lemma 2}
The number $w_D$ is odd.
\end{lemma}
\begin{proof}
Let $\ve$ be a section of $V$ and $h$ be a section of $C$. By \eqref{filt1} and \eqref{filtl}, the subspaces $L_i$ are spanned by vector fields $\ve, h,\ad_{\ve}h,\ldots,\ad_{\ve}^ih$. Assume that $[\ad^s_{\ve}h,\ad^{w-s-1}_{\ve}h]\in H_{k+1}$ for any $s=0,\ldots,\lfloor \frac{w-1}{2}\rfloor$ on an open set of the fiber of $\mathbb P(D^\perp)$. Applying $\ad_{\ve}$ and using the Jacobi identity we get
%new set of relations.
\begin{equation}
\label{addrel}
[\ad^{s+1}_{\ve}h,\ad^{w-s-1}_{\ve}h]+[\ad^s_{\ve}h,\ad^{w-s}_{\ve}h]\in H_{k+1}
\end{equation}
Assume that $w$ is even. Then substituting $s=\frac{w}{2}-1$ into \eqref{addrel}, we get that
$[\ad^{\frac{w}{2}-1}_{\ve}h,\ad^{\frac{w}{2}+1}_{\ve}h]\in H_{k+1}$. Then using \eqref{addrel} consecutively we get
$[\ad^s_{\ve}h,\ad^{w-s}_{\ve}h]=0\mod H_{k+1}$ for any $s=0,\ldots,\frac{w}{2}$. Therefore $w<w_D$ and thus $w_D$ can not be even.
\end{proof}

\begin{remark}
\label{0rem}
{\rm Note that from the similar arguments as in the previous lemma one can show that for every section $\ve$ of $V$ and $h$ of $C$ we have  $$[h,\ad^{w_D}_{\ve}h](\lambda)\equiv (-1)^s[\ad^s_{\ve}h,\ad^{w_D-s}_{\ve}h](\lambda)\,{\rm mod} H_{k+1}(\lambda), \quad \lambda \in \mathcal R_1.$$ This implies that
\begin{equation}
\label{dimA}
\dim A_{w_D}(\lambda)/H_{k+1}(\lambda)=1, \quad \lambda\in\mathcal R_1
\end{equation}$\Box$}
%$\mathbb R_1$
%if $\lambda\in \mathbb P(D^\perp)_{reg}$ then $[L_0, L_{w_D}](\lambda)\nsubseteq H_{k+1}(\lambda)$.
\end{remark}

\begin{lemma}\label{lemma 3}
If $i_D=1$ then $[D_{k+1},D_{k+1}]=D$.
%If $[D_{k+1},D_{k+1}]=D$ then $w_D\leq\frac{k+1}{2}$.
\end{lemma}
\begin{proof}
If $i_D=1$ then $K_1(\lambda)\subset[H_{k+1},H_{k+1}](\lambda)$ for every $\lambda\in P(D^\perp)$. It follows from the constructions that $$\spn\Bigl\{\displaystyle{\bigcup_{\lambda\in \mathbb P(D^\perp)\bigl(\pi(\lambda)\bigr)}\pi_*(K_1(\lambda))}\Bigr\}= D(\pi(\lambda))$$ (as a matter of fact, the curve $[p]\mapsto [\pi_*(K_1(p,q))/D_{k+1}(q)]$, $p\in D^\perp(q)$, is a rational normal curve in $\mathbb P\bigl(D(q)/D_{k+1}(q)\bigr)$).
Hence $[D_{k+1},D_{k+1}]=D$.
%If $[D_{k+1}, D_{k+1}]=D$ then TODO.
\end{proof}

\begin{lemma}\label{lemma 1}
If $w_D=1$ then $i_D\neq k-1$.
\end{lemma}
\begin{proof}
Let $\ve$ be a section of $V$ and $h$ be a section of $C$. Denote
$$
\ve_1=[h,\ve],\qquad \ve_2=[h,\ve_1].
$$
First of all notice that $\ve_2$ is a section of $H\setminus H_{k+1}$ since $w_D=1$. By \eqref{VHC} $[\ve,\ve_2]\in\Gamma(H)$. Our aim is to prove that $[\ve,\ve_2]\in K$ provided that $\ve_2\in K$. By definition of the spaces $K_i$,  this statement implies that $i_D\neq k-1$). Assume that $\ve_2\in K$ and let $Y=[h,[\ve,\ve_2]]$. We will prove that $Y\in\Gamma(H)$.

If $Y\notin\Gamma(H)$ then $Y$ spans $\Lambda$ modulo $H$. We will show that it leads to the contradiction. The Jacobi identity and \eqref{VHC} imply that
$$
[\ve_1,\ve_2]=[h,[\ve,\ve_2]]=Y\mod H,
$$
since $[h,e_2]=0\mod H$. Moreover we have
$$
[\ve_1,[\ve,\ve_2]]=[[h,\ve],[\ve,\ve_2]]=-[\ve,Y] \mod\Lambda
$$
and from above we obtain
$$
[\ve_2,[\ve,\ve_1]]=[\ve_1,[\ve,\ve_2]]-[\ve,[\ve_1,\ve_2]]=-2[\ve,Y] \mod\Lambda.
$$
We use the Jaccobi identity once again and we get
$$
[h, [\ve_1,[\ve,\ve_1]]]=[\ve_2,[\ve,\ve_1]]+ [\ve_1,[\ve,\ve_2]]=-3[\ve, Y] \mod\Lambda.
$$
On the other hand $[\ve_1,[\ve,\ve_1]]\in\Gamma(H)$ and therefore $[h, [\ve_1,[\ve,\ve_1]]]=cY$ for some function $c$. We conclude that
$$
[\ve,Y]=0\mod\Lambda.
$$
But it means $\ve$ is a Cauchy characteristic vector field of $\Lambda$, i.e. $[\ve,\Lambda]\subset\Lambda$. It implies that $\ve$ is a section of $C$, which is not the case.
 %is constant in the fiber of $P(D^\perp)$ which is not true.
 Thus we get the contradiction and the proof is completed.
\end{proof}

\begin{corollary}
\label{cor 2}
If $D$ is a $(5,7)$-distribution with maximal Kronecker index then $i_D=2$ or $D_3$ is integrable.
\end{corollary}
\begin{proof}
Let $h,\ve,\ve_1,\ve_2$ be as in the proof of Lemma \ref{lemma 1}.
%If $D_3$ is not integrable then
%$[h,\ve_1]\notin\Gamma(H_3)$ i.e.
%then $w_D=1$.
%Indeed, assume the converse, i.e.
First let us prove that if $D_3$ is not integrable then $w_D>1$ then $D_3$ is integrable. Assume the converse, i.e. that
\begin{equation}
\label{he1H3}
[h,\ve_1]\in\Gamma(H_3).
\end{equation}
Applying $\ad_\ve$ to the last relation and using Jacobi identity, we get that
\begin{equation}
\label{he1e1H3}
[h,[\ve,\ve_1]]\in \Gamma(H_3).
\end{equation}
Applying  $\ad_\ve$ and the Jacobi identity once more we get that
\begin{equation}
\label{ade2}
-[\ve_1,[\ve,\ve_1]]+[h,[\ve,[\ve,\ve_1]]]\in \Gamma(H_3)
\end{equation}
On the other hand, $[\ve,[\ve,\ve_1]]\in H_3={\rm span}\{h, \ve,\ve_1,[\ve,\ve_1]\}$, which together with \eqref{he1H3} and \eqref{he1e1H3} implies that $[h,[\ve,[\ve,\ve_1]]]\in \Gamma(H_3)$. Hence \eqref{ade2} implies that $[\ve_1,[\ve,\ve_1]]\in \Gamma(H_3)$. So, $H_3$ is integrable and therefore $D_3$ is integrable. We get the contradiction.

Thus if $D_3$ is not integrable, then $w_D=1$.  Since $k=2$,  by the previous
lemma $i_D\neq 1$ therefore $i_D$ has to be equal to $2$. This completes the proof of the lemma.
\end{proof}

Further let $\widetilde {\mathcal R_2}$ be a subset of $\mathcal R_1$ consisting of all points $\lambda$ such that for any $r$ the dimension of subspaces $A_r$ is constant in a neighborhood of $\lambda$ (in $\mathcal R_1$). Obviously, $\widetilde {\mathcal R_2}$ is an open and dense subset of $\mathbb P(D^\perp)$.
Now we will assume that the condition (G2) holds, i.e. $D_{k+1}^2=D$. Then for any $\lambda$ there exists an integer $r$ such that
$K_1(\lambda)\subset \mathcal A_r(\lambda)$. Clearly $r(\lambda)\geq w(\lambda)$. Let
\begin{equation}
\label{rl}
r(\lambda)=\min\{r: K_1(\lambda)\subset \mathcal A_r(\lambda)\}.
\end{equation}
Note that the function $r(\lambda)$ is lower semicontinuous on the set $\widetilde {\mathcal R_2}$. Therefore, if $ \mathcal R_2$ denotes a subset of $\widetilde{\mathcal R_2}$ consisting of all points $\lambda$ such that $r(\lambda)$ is constant in a neighborhood of $\lambda$, then $\mathcal R_2$ is open and dense in $\mathbb P (D^\perp)$. Note that  the intersection of $\mathcal R_2$ with any fiber of $\mathbb P(D^\perp)(q)$ for any $q\in \pi(\mathcal R_2)$ is an open set in the Zariski topology of this fiber.

%Finally for any $q\in \pi( \widetilde {\mathcal R_2})$ denote
%$$r_D(q)=\max\{ r(\lambda):\lambda\in \mathbb P (D^\perp)(q)\}.$$
%%of $h\in\Gamma(C)$
%Then $r(\lambda)=r_D\bigl(\pi(\lambda)\bigr)$ for any $\lambda\in\mathcal R_2$.
%The numbers $r_D(q)$ are integer-valued invariants of the distribution $D$ at $q$.
%A regular point $q\in $
%is said to be \emph{ strongly regular} if $r_D$ is constant in a neighborhood of $q$.
%It is clear that
%the set of strongly regular points is open and dense subset of $M$.
%If $r_D(q)\equiv r_D$ for some integer $r_D$ then
%There exists an open and dense
% Let $\mathcal R_2 =\{\lambda\in\widetilde{\mathcal R_2}: r(\lambda)= r_D(\pi(\lambda))\}$.
 %Then $\mathcal R_2$ is open and dense subset of $\mathbb P(D^\perp)$. Moreover,

\section{Construction of canonical frames}
\setcounter{equation}{0}
\setcounter{theorem}{0}
\setcounter{lemma}{0}
\setcounter{proposition}{0}
\setcounter{corollary}{0}
\setcounter{remark}{0}
%Let us assume that $D$ has maximal Kronecker index.

Now we formulate and prove our main result on the frames for distributions from the considered class.
%with completely non-holonomic $D_{k+1}$.
\begin{theorem}\label{thm 2}
Assume that a $(2k+1,2k+3)$-distribution $D$ with $k>1$ has the maximal first Kronecker index and the square of the subdistribution $D_{k+1}$ is equal to the distribution $D$.
Let $\mathcal R_1$ and $\mathcal R_2$ be the open dense subsets of $\mathbb P(D^\perp)$ defined in the previous sections.

 %distribution with constant integer-valued invariants $w_D(q)\equiv w_D$ and $i_D(q)\equiv i_D$.
\begin{enumerate}
%\item If $w_D\equiv \frac{k+1}{2}$ and $i_D\equiv 1$, then there exists a canonical frame on rank 3 bundle over $\mathcal R_1$;
\item If $w_D$
%is a constant, which
is not equal to $\frac{k+1}{2}$, and $i_D\equiv 1$, then there exists a canonical frame on rank 2 bundle over $\mathcal R_1$;
\item If $w_D\equiv \frac{k+1}{2}$ and $i_D\equiv 1$, then there exists a canonical frame on rank 3 bundle over $\mathcal R_1$;
\item If $i_D$ is
%constant and
greater than $1$ then there exists a canonical frame on rank 1 bundle over a neighborhood  of any point of $\mathcal R_2$.
%\in P(D^\perp)_{reg}$.
\end{enumerate}
Two corank 2 distributions $D$ and $D'$ satisfying conditions (G1) and (G2) are equivalent if and only if there is a diffeomorphism (of the corresponding bundles) sending the canonical frame of $D$ to the canonical frame of $D'$.

\end{theorem}

Note that Lemma \ref{lemma 2} implies that $w_D=\frac{k+1}{2}$ does not hold unless $k\equiv 1\mod 4$.
 If we take into account Corollary \ref{cor 2}, then Theorem \ref{thm 2} implies immediately the following
 %we get that if $D$ is $(5,7)$-distributions with maximal Kronecker index and non-integrable $D_{k+1}$ then always $i_D=2$ and additionally $[D_{k+1},D_{k+1}]=D$. In particular $D_{k+1}$ is completely non-holonomic, because $[D,D]=TM$.
\begin{corollary}\label{cor 3}
 Assume that a $(2k+1,2k+3)$-distribution $D$ has the maximal first Kronecker index and the square of the subdistribution $D_{k+1}$ is equal to the distribution $D$. Then the dimension of its algebra of infinitesimal symmetries does not exceed
 \begin{enumerate}
 \item $2k+6$, if $k\nequiv 1 \mod 4$ and $k>2$;
 \item $2k+7$, if $k\equiv 1\mod 4$ and $k>1$;
 \item $9$, if $k=2$.
 \end{enumerate}
 \end{corollary}
 In section 4 we show that the upper bounds for the algebra of infinitesimal symmetries from the previous Corollary are sharp and describe all corank 2 distributions from the considered class for which these upper bounds are attained.
 %there exists a canonical frame on rank 3 bundle over $P(D^\perp)$ if $k\equiv 1\mod 4$ and there exists a canonical frame on %rank 2 bundle over $P(D^\perp)$ if $k\nequiv 1\mod 4$.

%\end{corollary}

%\begin{corollary}\label{cor 4}
%If $D$ is a $(5,7)$-distribution with non-integrable $D_{k+1}$ then there exists a canonical frame on $D^\perp_{reg}$.
%\end{corollary}

{\bf Proof of Theorem \ref{thm 2}}
 Recall that $V$ and $C$ are  rank 1 distributions on $\mathbb P(D^\perp)$.  Let $V_0$ and $C_0$ denote the corresponding bundles with zero section removed. Obviously, they are principal $\mathbb R^*$-bundles, where $\mathbb R^*$ is the multiplicative group of real numbers. Further, recall that the fiber $D^\perp (q)$ of $D^\perp$ over  a point $q\in M$ is a plane and  the fiber of $\mathbb P(D^\perp(q))$ is a projective line. Fix a point $\lambda=(p,q)\in \mathbb P(D^\perp)$ and consider all homogeneous coordinates $[x_1:x_2]$ on $\mathbb P(D^\perp(q))$ such that the point $[p]$ is equal to $[1:0]$ in these coordinates.
Let $B$ denote the rank 2 bundle over $\mathbb P(D^\perp)$ with the fiber over $\lambda$ consisting of all such homogeneous coordinates on $\mathbb P(D^\perp(q))$. In other words, the fiber of $B$ over $\lambda=(p,q)$ is the set of all projective mappings from $\mathbb {RP}^1$ to
 $\mathbb P(D^\perp(q))$) sending the point $[1:0]$ to $[p]$.
Obviously, $B$ is a principal $ST(2,\mathbb R)$-bundle, where $ST(2,\mathbb R)$ is the group of  $2\times 2$ upper triangular matrices with the determinant $1$.
 %space of $1$-jets of restrictions of sections of $V_0$ to the integral curves of the distribution $V$ (i.e. to the fibers of the %bundle $\mathbb P(D^\perp)$).
 The following 2 bundles over $\mathbb P(D^\perp)$ play an important role in the sequel
$$
%B_1=C_0,\qquad
B_1=V_0\times C_0,\qquad B_2=B\times C_0.
$$
Here $B_1$ is the  bundle over $\mathbb P(D^\perp)$ with the fibers equal to the Cartesian product of the corresponding
fibers of $V_0$ and $C_0$; the bundle $B_2$ is understood similarly.
Obviously,
%$B_1$ is a principal $\mathbb R^*$-bundle,
$B_1$ is a principal $\mathbb R^*\times \mathbb R^* $-bundle, and $B_2$ is a principal $T(2,\mathbb R)$-bundle, where $T(2,\mathbb R)$ is the group of  $2\times 2$ upper triangular matrices.
%$G\oplus GL(1)$-principal bundle, where
%$$
%G=\left\{\left(\begin{array}{cc}1 & \tilde{a}\\ 0 & a\end{array}\right)\in GL(2)\right\}.
%$$

The group actions define fundamental vector fields on bundles $B_1$ and $B_2$. Let us choose bases in the space of fundamental vector fields as follows.
%distinguish fundamental vector fields on each of these bundles,
%$B_1$, $i=1,2,3$,
%which will constitute the bases for the tangent spaces to the fibers at every point.

First, let  $\bb$ denote the vector field on
%$B_1$ ($=
$C_0$
%)
generating the flow $(\lambda,h)\mapsto (\lambda, e^s h)$,
for any $(\lambda,h)\in C_0$, where $\lambda\in \mathbb P(D^\perp)$, $h\in C_0(\lambda)$. Since the fibers of $C_0$ appear as  factors for the fibers of $B_1$ and $B_2$ we can define the analogous vector field on these bundles as well (just by defining the corresponding flow such that it acts trivially on the remaining factors). We will denote this vector field on $B_1$ and $B_2$ by $\bb$ as well.

Secondly, let $\ba$ denote the vector field on $B_1$ generating the flow
$$\bigl(\lambda,(\ve,h)\bigr)\mapsto \bigl(\lambda,(e^s \ve, h)\bigr)$$
for any $\bigl(\lambda,(\ve,h)\bigr)\in B_1$, where $\lambda\in \mathbb P(D^\perp)$, $\ve\in V_0(\lambda)$, $h\in C_0(\lambda)$.
By the same letter $\ba$ denote the vector field on $B_2$ generating the flow $$\bigl(\lambda, ([x_1:x_2], h)\bigr)\mapsto \bigl(\lambda, ([x_1: e^{-s} x_2],h)\bigr),$$ where $\lambda=(p,q) \in \mathbb P(D^\perp)$,  $[x_1:x_2]$ are homogeneous
coordinates on $\mathbb P(D^\perp(q))$ such that $[p]=[x_1:0]$, and $h\in C_0(\lambda)$.

Finally, let $\bc$ denote the vector field on $B_2$ generating the flow
$$\bigl(\lambda, ([x_1:x_2], h)\bigr)\mapsto \bigl(\lambda, ([x_1-s x_2: x_2],h)\bigr),$$ where $\lambda$,  $[x_1:x_2]$, and $h$ are as above.

It is easy to show that we have the following relations on $B_2$

\begin{equation}
\label{abc}
[\ba,\bb]=0,\quad [\bc, \bb]=0, \quad [\ba,\bc]=-\bc,
\end{equation}
and the first relation holds on $B_1$ as well.
%$\ve, h\neq 0$.

%Now let us briefly sketch the main steps of our proof.
%First in the case of $i_D>1$ one can distinguish the canonical section $\ve$ of $V_0$. Then in all cases

%Assume that $e_0$ spans $V$ and $h$ spans $C$. $h$ and $e_0$ are given up to the transformations:
%$$
%e_0\mapsto ae_0,\qquad h\mapsto bh.
%$$
%Then $(b)$ can be considered as a coordinate function in fibres of $B_1$, the pair $(a,b)$ can be considered as coordinate functions %in fibres of $B_2$, and  the triple $(a,e_0(a),b)$ can be considered as coordinate functions in fibres of $B_3$. We will construct %partial connections on the bundles $B_1$, $B_2$ and $B_3$ in the directions of $e_0$ and $h$. Then we will lift vector fields $ae_0$ %and $bh$ to the corresponding points in the bundles. Finally we will construct the frame.
Let $\Pi_i: B_i\rightarrow \mathbb P(D^\perp)$ be the canonical projection.
We say that a vector field $E$ on the bundle $B_1$ is a \emph{lift of the distribution $V$ to $B_1$},  if for any $\mu_1=(\lambda,(\ve,h)\bigr)\in B_1$, where $\lambda\in \mathbb P(D^\perp)$, $\ve\in V_0(\lambda)$, $h\in C_0(\lambda)$, one has
\begin{equation*}
%\label{B1E}
(\Pi_1)_*E(\mu_1)=\ve.
\end{equation*}
To define the analogous notion on the bundle $B_2$ first define it on the bundle $B$. Take $\mu=\bigl(\lambda, [x_1:x_2]\bigr)\in B$ %or $\mu=\bigl(\lambda, ([x_1:x_2], h)\bigr) \in B_2$,
where $\lambda=(p,q) \in \mathbb P(D^\perp)$ and   $[x_1:x_2]$ are homogeneous
coordinates on $\mathbb P(D^\perp(q))$ such that $[p]=[x_1:0]$.
%, and $h\in C_0(\lambda)$.
Then $t=\frac{x_2}{x_1}$ defines coordinates on $\mathbb P(D^\perp(q))$ and a \emph{lift of the distribution $V$ to the bundle $B$} is a vector field $E$ on $B$, satisfying  the following relation for any such $\mu$:
\begin{equation*}
%\label{BE}
(\Pi)_*E(\mu)={\partial\over \partial t}([p]),
\end{equation*}
where $\Pi:B\rightarrow \mathbb P(D^\perp)$ is the canonical projection. Finally a \emph{ lift  of the distribution $V$ to the bundle $B_2$} is a vector field on $B_2$ such that $(\mathfrak P)_* E$ is a lift of $V$ to $B$, where $\mathfrak P:B_2\rightarrow B$ is the canonical projection.

To define the lift of the distribution $C$ to the bundle $B_i$ first define the lift of $C$ to $C_0$: it is a vector field $\mH$ on $C_0$ such that if $\widetilde \Pi:C_0\mapsto \mathbb P(D^\perp)$, then for any $(\lambda,h)\in C_0$, where $\lambda\in \mathbb P(D^\perp)$, $h\in C_0(\lambda)$, one has $(\widetilde \Pi)_*\mH\bigl((\lambda,h)\bigr)=h$. Then  the  vector field $\mH$ on the bundle $B_i$, $i=1,2$ is called a \emph{lift of the distribution $C$ to $B_i$} if $(\mathfrak P_i)_*\mH$
is a lift of $C$ to $C_0$, where $\mathfrak P_i:B_i\rightarrow C_0$ is the canonical projection.
%$\mu_i\in B_i$ one has
%\begin{equation}
%\label{B1E}
%(\Pi_i)_*\mH(\mu_i)=h.
%\end{equation}

Now let $W_i$, $i=1,2$ be the distribution of tangent spaces to the fibers of $B_i$, i.e. $$W_i:={\rm ker}(\Pi_i)_*.$$ Distributions $W_i$ are also called the \emph{vertical distribution on $B_i$}.
Lifts $E$ and $\mH$ are defined modulo vertical distributions $W_i$. By constructions, all lifts $E$ and $\mH$ of $V$ and $C$, respectively, satisfy the following relations
\begin{equation}
\label{EHabs}
\begin{split}
~&[\ba, E]=E \,\, \mod W_i,\quad [\bb, E]\in \Gamma(W_i),\quad [\bc, E]\in \Gamma(W_2),\\
~&[\bb, \mH]=\mH \,\mod W_i,\quad [\ba, \mH]\in \Gamma(W_i),\quad [\bc, \mH]\in \Gamma(W_2)\\
\end{split}
\end{equation}
Here the formulas containing $\bc$ are related to the bundle $B_2$ only.

%(i.e. vector fields tangent to the fibers of the corresponding $B_i$).
Our goal is to choose the lifts $E$ and $\mH$ in a canonical way. Once it is done one can complete the tuple consisting of the fundamental vertical vector fields and the canonical lifts to the canonical frame on the corresponding bundle $B_i$ by taking appropriate iterative Lie brackets of
these canonical lifts.

%Given a rank 1 distribution $\mathcal X$ on $\mathbb P (D^\perp)$ we say that a rank 1 distribution $\widehat{\mathcal X}$ on the %bundle $B_i$ defines a \emph{partial connection on $B_i$ in the direction of $X$} if $(\Pi_i)_*\widehat X=X$. A partial connection %is called \emph{principal} if it preserved by the action of the structure group.
%
%Now let us briefly sketch our proof. We will construct canonical principal partial connections $\widehat  V$ and $\widehat C$ in the %directions of $V$ and $C$ on one of the bundles $B_i$. Using special features of the bundles $B_i$ we will get from this the %canonical vector fields (sections) $E_0$ and $H$  of $\widehat V$ and $\widehat C$ respectively.

In the sequel $\widehat V$, $\widehat L_j$, $\widehat \mH_{k+1}$, $\widehat K_j$, $\widehat{\mathcal A}_j$, $\widehat H$, and $\widehat \Lambda$ denote the pull backs of distributions $V$, $L_j$, $H_k+1$, $K_j$, $\mathcal A_j$, $H$, and $\Lambda$, respectively, to the corresponding $B_i$, $i=1,2$.

{\bf Step 1.  The canonical lift of $V$.}
%Partial connection in direction of $V$.}
%Choose sections $e_0$ and $h$ of $V$  and $C$, respectively.
First we will work on the bundle $B_1$.
Let $E$ be a lift of $V$ and $\mH$ be a lift of $C$ to $B_1$.
By constructions, vector fields $E, \mH, \ad_E \mH,\ldots,\ad^i_E\mH$ span $\widehat L_i$
modulo $W_1$
and $\widehat L_k=\widehat H_{k+1}$. It implies that $\ad^{k+1}_E\mH\in\Gamma(\widehat H_{k+1})$.
Therefore there exist a function $\eta$ such that
\begin{equation}
\label{mu}
\ad^{k+1}_E \mH\equiv
%\sum_{i=1}^k\
\eta\, \ad^k_E \mH \quad {\rm mod}\,\widehat L_{k-1}
\end{equation}

First, we are looking for a pair of lifts $E$ and  $\mH$ satisfying the condition
\begin{equation}\label{cond 1}
\ad_{E}^{k+1}\mH\equiv 0 \mod\,\widehat L_{k-1}
\end{equation}
For this start with some lift $E$ of $V$ and $\mH$ of $C$ and assume that they satisfy \eqref{mu} for some function $\eta$.
Take other lifts $\widetilde E$ and $\widetilde \mH$.
Then there exist functions $\alpha$, $\beta$,$\gamma$, $\delta$ such that
\begin{equation}
\label{transf}
\widetilde E=E+\alpha \ba +\beta \bb,\quad
\widetilde \mH=\mH+\gamma\ba+\delta\bb.
\end{equation}
By direct computations, using relations \eqref{EHabs}, one gets

\begin{equation}
\label{tildemu}
\ad^{k+1}_{\widetilde E} \widetilde \mH\equiv
%\sum_{i=1}^k\
\bigl(\eta+(k+1)(\frac{k}{2}\alpha+\beta)\bigr) \ad^k_{\widetilde E} \widetilde \mH \quad {\rm mod} \widehat L_{k-1}
\end{equation}
Thus, a pair of lifts $\widetilde E$ and $\widetilde \mH$ satisfies condition \eqref{cond 1} if and only if
\begin{equation}
\label{cond11}
\frac{k}{2}\alpha+\beta=-\frac{\eta}{k+1}.
\end{equation}

Further,
%let $(B_1)_{reg}=
%$\Pi_1^{-1}\bigl(\mathbb P(D^\perp)_{reg}\bigr)$.
from  Remark \ref{0rem} it follows that  $[L_0, L_{w_D}](\lambda)\nsubseteq H_{k+1}(\lambda)$ for $\lambda\in \mathcal R_1$. Hence there is a section $G$ of $\widehat K_1$ on $\Pi_1^{-1}(\mathcal R_1)$, unique modulo $\widehat H_{k+1}$, such that for any lifts $E$ and $\mH$ of $V$ and $C$ one has
\begin{equation}
\label{Geq}
\ad_E^{i_D-1}G\equiv [\mH,\ad_{E}^{w_D}\mH]\,\,\mod\,  \widehat K_{w_D-1}.
\end{equation}
Now assume that $\mu, \tilde \mu\in \Pi_1^{-1}(\mathcal R_1) $,
\begin{equation}
\label{mm}
\mu=\bigl(\lambda,(\ve,h)\bigr),\quad  \tilde \mu=\bigl(\lambda,(a\ve,b h)\bigr),
 \end{equation}
 where $\ve\in V_0(\lambda)$, $h\in C_0(\lambda)$, and $a,b\in \mathbb R^*$. Then from \eqref{Geq} it follows immediately that
%If $\widetilde {e_0}$ and $\widetilde h$ are as in \eqref{ab}, then the corresponding section $\widetilde g$ of $K_1$ satisfies

\begin{equation}
\label{Gtrans1}
(\Pi_1)_*G(\widetilde \mu)\equiv a^{w_D-i_D+1}b^2 (\Pi_1)_*G(\mu)\mod \widehat H_{k+1}.
\end{equation}

Assume that $\lambda\in\mathcal R_2$ and $r(\lambda)\equiv r$ in a neighborhood $\widetilde U$ of $\lambda$. Choose a local basis of $\widehat{\mathcal A}_{r-1}$ in $\Pi_1^{-1}$ and complete it to a local basis of $\widehat{\mathcal A}_r$ by a tuple of vector fields $\bigl\{[\ad_{E}^s \mH, \ad_{E}^{r-s} \mH]\bigr\}_{s\in\mathcal S}$, where $\mathcal S\subset\{0,\ldots,r\}$. Since by \eqref{rl} $G$ is a section of $\widehat {\mathcal A}_r$ but does not belong to $\widehat {\mathcal A}_{r-1}$ , there exists $\bar s\in\mathcal S$ such that the coefficient $c_{\bar s}$ near one of the field $[\ad_{E}^{\bar s} \mH, \ad_{E}^{r-\bar s} \mH]$ in the expansion of $G$ in the chosen basis does not vanish at any point of $\mathcal R_2$ over a neighborhood $U\subset\widetilde U$ of $\lambda$. Let $\mathcal U=\pi^{-1}(U)$.
%Let
If points $\mu, \tilde\mu\in\mathcal U$ are related as in \eqref{mm}, then using \eqref{Gtrans1} it is easy to see that
\begin{equation}
\label{cs}
c_{\bar s}(\tilde \mu)=a^{r-w_D+i_D-1} c_{\bar s}(\mu)
\end{equation}
Note that by constructions $r\geq w_D$. So, if $i_D>1$ then the power of $a$ in the transformation rule \eqref{cs} is positive. So, we can distinguish the codimension 1 submanifold $B_3$ of $\Pi_1^{-1}(\mathcal U)$, consisting of all points of  $\Pi_1^{-1}(\mathcal U)$ with $c_{\bar s}=1$ if $r-w_D+i_D$ is even and with $|c_{\bar s}|=1$ if $r-w_D+i_D$ is odd. As a matter of fact $B_3$ is a $R^*$-bundle over $\mathcal U$, which is a reduction of $B_1$. One can naturally identify $B_3$ with $C_0$ (over $\mathcal U$).
Now we can consider only lifts of $V$ and $C$ which are tangent to $B_3$ or shortly lifts of $V$ and $C$ to $B_3$. If  $E$ and $\widetilde E$ are lifts of $V$ to $B_3$ and $\mH$ and $\widetilde\mH$ are lifts of $C$ to $B_3$ then instead of transformation rule
\eqref{transf} we have
\begin{equation}
\label{transfred}
\widetilde E=E+\beta \bb,\quad
\widetilde \mH=\mH+\delta\bb.
\end{equation}
So, the normalization condition \eqref{cond11} transforms to the condition $\beta=-\frac{\eta}{k+1}$ and gives the canonical lift of $V$ to $B_3$.

 %be the subspace
%The fact that $D_{k+1}$ is bracket generating implies that $H_{k+1}$ is bracket generating as well.
% Since $H_{k+1}$ is generated by $e_0$ and $h$ we know that $g$ is a combination of iterated Lie bracket of $e_0$ and $h$.
% Let $r_D$ be the smallest number such that the $r_D$th power $H_{k+1}^{r_D}$ of $H_{k+1}$ contains $g$ as a section.
%  and let $s_D$ be the smallest number of appearance of $e_0$ in the Lie brackets which allows to present $g$ as a section of $H_{k+1}^{(r_D)}$. It follows from the definition that $s_D\geq w_D$ and $r_D\geq 2$.
%
%If $i_D>1$ then $s_D>w_D$. Let $c$ be the coefficient at the term of order $s_D$ in $e_0$ and of order $r_D$ in $h$ in the presentation of $g$ as a section of $H_{k+1}^{(r_D)}$. It follows that
%
%if $e_0\mapsto ae_0$ and $h\mapsto bh$. Hence, for a given $b$, we can choose $a$ in such a way that $c\equiv 1$. In this way we normalise $e_0$ and equation \eqref{eq 1} reduces to the first order equation for one unknown function $b$. In other words we get a partial connection on the bundle $B_1$ in the direction $e_0$.

On the other hand, if $i_D=1$ then by definitions $r=w_D$. Therefore from \eqref{cs} it follows that $c_{\bar s}$ is constant on the fibers of $B_1$ (actually it is identically equal to $1$) and we cannot make the above reduction of the bundle $B_1$. Instead, we are looking for an additional condition for the lifts to $B_1$.
Again fix some lift $E$ and $\mH$ to $B_1$ of $V$ and $C$ respectively and $G$ is a vector field defined by \eqref{Geq} modulo $\widehat H_{k+1}$.
By constructions, $\widehat K_i=\widehat H_{k+1}+{\rm span} \{G, \ad_E G,\ldots,\ad^{i-1}_E G\}$ and
and $\widehat K_{k}=\widehat H$. It implies that $\ad^{k}_E G\in\Gamma(\widehat H)$.
Therefore there exist a function $\upsilon$ such that
\begin{equation}
\label{mu1}
\ad^{k}_E G\equiv
%\sum_{i=1}^k\
\upsilon\, \ad^{k-1}_E G \quad {\rm mod}\,\widehat K_{k-1}
\end{equation}

We are looking for a pair of lifts $E$ and  $\mH$  such that
\begin{equation}\label{cond 2}
\ad_{E}^{k}G\equiv 0 \mod\,\widehat K_{k-1}
\end{equation}
For this as before take some pair of lifts $E$ and $\mH$ and assume that they satisfy \eqref{mu1} with some function $\upsilon$.
Take other lifts $\widetilde E$ and $\widetilde \mH$. Then relation \eqref{transf} holds for some functions $\alpha$, $\beta$,$\gamma$, $\delta$.
%\begin{equation}
%\label{transf}
%\widetilde E=E+\alpha \ba +\beta \bb,\quad
%\widetilde \mH=\mH+\gamma\ba+\delta\bb.
%\end{equation}
By direct computations, using relations \eqref{EHabs} and \eqref{Gtrans1}, one gets

\begin{equation}
\label{tildemu}
\ad^{k}_{\widetilde E} G\equiv
%\sum_{i=1}^k\
\Bigl(\upsilon+k\bigl(\bigl(\frac{k-1}{2}+w_D\bigr)\alpha+2\beta\bigr)\Bigr) \ad^{k-1}_{\widetilde E} G\quad {\rm mod} \widehat K_{k-1}
\end{equation}
Thus, a pair of lifts $\widetilde E$ and $\widetilde \mH$ satisfies condition \eqref{cond 2} if and only if
\begin{equation}
\label{cond21}
\bigl(\frac{k-1}{2}+w_D\bigr)\alpha+2\beta=-\frac{\upsilon}{k}.
\end{equation}

%Similarly to condition \eqref{cond 1}, we introduce the condition:
%\begin{equation}\label{cond 2}
%\ad_{e_0}^kg=0\mod K_{k-1}.
%\end{equation}
%If $e_0\mapsto ae_0$ and $h\mapsto bh$ then $g\mapsto a^{w_D}b^2g$ and:
%$$
%\ad_{e_0}^kg\mapsto a^{k+w_D}b^2\ad_{e_0}^kg + ka^{k-1}\left(ae_0(b^2a^{w_D})+\frac{k-1}{2}b^2a^{w_D}e_0(a)\right) \ad_{e_0}^{k-1}g\mod K_{k-1}.
%$$
%We see that both \eqref{cond 1} and \eqref{cond 2} can be satisfied at one time. The freedom of choice of $a$ and $b$ in \eqref{cond 2} is given by the equation:
%\begin{equation}\label{eq 2}
%2ae_0(b)+\left(\frac{k-1}{2}+w_D\right)be_0(a)=0.
%\end{equation}
We see that linear equations \eqref{cond11} and \eqref{cond21} (w.r.t. $\alpha$ and $\beta$) are linearly independent if and only if $w_D\neq\frac{k+1}{2}$.
Hence in the case $i_D=1$ and $w_D\neq\frac{k+1}{2}$ conditions \eqref{cond 1} and \eqref{cond 2} fix uniquely the lift of $V$ to the bundle $B_1$.

 It remains to consider the case $i_D=1$ and $w_D=\frac{k+1}{2}$. In this case in general $V$ cannot be lifted to $B_1$ canonically,  but one can find the canonical lift of $V$ to $B_2$.
 First define the canonical lift of $V$ to the bundle $B$.
 Take $\mu=\bigl(\lambda, [x_1:x_2]\bigr)\in B$
where $\lambda=(p,q) \in \mathbb P(D^\perp)$,  $[x_1:x_2]$ are homogeneous
coordinates on $\mathbb P(D^\perp(q))$ such that $[p]=[x_1:0]$.  Then $\varphi([p])=\frac{x_2}{x_1}$ defines coordinates on $\mathbb P(D^\perp(q))$. Consider the curve
\begin{equation}
\label{ups}
\Upsilon_\mu(t)=\bigl((\varphi^{-1}(t),q),[x_1:(x_2-tx_1)]\bigr).
\end{equation}
Then the  canonical lift $E$ of $V$ to $B$ is defined by
\begin{equation}
\label{canEB}
E(\mu)=\frac{d}{dt}\Upsilon_\mu(t)|_{t=0}.
\end{equation}

Now we are ready to define the canonical lift of $V$ to $B_2$. For this let as before $\mathfrak P:B_2\rightarrow B$ be the canonical projection and consider all lifts $E$ of $V$ to $B_2$ such that $\mathfrak P_*(E)$ is the canonical lift of $V$ to $B$. If $E$ and $\tilde E$ are two such lifts then they are related as in \eqref{transfred} for some function $b$. By analogy with above the normalization condition \eqref{cond11} transforms to the condition $\beta=-\frac{\eta}{k+1}$ and gives the canonical lift of $V$ to $B_2$. By this we have completed to lift $V$ to the corresponding bundles $B_i$ in all possible cases.

Note that by direct computation one has that the canonical lift $E$ to $B_2$ of $V$ satisfies the following relations:

\begin{equation}
\label{adrel}
[\ba, E]=E, \quad [\bb, E]=0, \quad [\bc, E]=-2\ba.
\end{equation}
Note also that the first two relations are valid for the canonical lift of $V$ to $B_1$ as well. For this, using \eqref{EHabs}, it is enough to show that the line distribution generated by the canonical lift $E$ is invariant with respect to the flows generated by the vector fields $\ba$ and $\bb$. The latter follows from the normalization conditions \eqref{cond 1} and \eqref{cond 2} and the fact that the distribution $\widehat L_{k-1}$ is invariant w.r.t. to these flows.

Relations \eqref{abc} and \eqref{adrel} imply that the vector fields $\ba,\bc, E$ form the Lie algebra isomorphic to $\mathfrak {sl}_2(\mathbb R)$, and together with $\bb$ they form the Lie algebra isomorphic to $\mathfrak {gl}_2(\mathbb R)$.

%  then in general there is no other possibility to get additional, independent, first order equation for functions $a$ and $b$ in the direction of $e_0$. Therefore we introduce the condition of second order:
%\begin{equation}\label{cond 3}
%\ad_{e_0}^{k+1}h=0\mod L_{k-2}.
%\end{equation}
%This is an extension of \eqref{cond 1} and it defines a canonical projective structure on $e_0$. Namely, the freedom of choice of $a$ is given by:
%\begin{equation}\label{eq 3}
%\mathbb{S}(a)=0,
%\end{equation}
%where $\mathbb{S}(a)=2e_0^2(a)-(e_0(a))^2$ is Schwartz derivative. Equations \eqref{eq 1} and \eqref{eq 3} defines a partial connection on $B_3$ in the direction of $e_0$.

{\bf Step 2. The canonical lift of $C$.}
We assume that $E$ is the canonical lift of $V$ to the corresponding bundle $B_i$ defined in Step 1 and $\mathfrak H$ is a lift of $C$ to the same $B_i$. As before, let $G$ be a section of $K_1$ satisfying \eqref{Geq}. Define
$$
F=[E,[\mH,\ad_E^{k-1}G]].
$$
Then $F$ is a vector field not contained in $\hat\Lambda$. Indeed $\ad_E^{k-1}G$ is out of $\hat K$ and thus $[\mH,\ad_E^{k-1}G]$ is out of $\hat H$, but in $\hat \Lambda$. Then $[E,[\mH,\ad_E^{k-1}G]]$ is out of $\hat\Lambda$ since $[V,\Lambda]=T\mathbb P(D^\perp)$. There exists a function $\xi_0$ such that
\begin{equation}
\label{xi0}
\ad_\mH F\equiv\xi_0 F\mod \hat\Lambda.
\end{equation}

We are looking for a lift $\mH$ of $C$ (to one of the bundles $B_i$) satisfying:
\begin{equation}\label{cond 4}
\ad_\mH F\equiv 0\mod \hat\Lambda.
\end{equation}
For this start with some lift $\mH$ to $B_1$ or $B_2$ and assume that it satisfy \eqref{xi0} for some function $\xi_0$. Take another lift $\widetilde \mH$ of $V$. Then in the case of a lifting to  $B_1$ there exist functions $\gamma$ and $\delta$ such that
\begin{equation}
\label{hB1}
\widetilde \mH=\mH+\gamma \ba+\delta \bb,
\end{equation}
while in the case of a lifting to $B_2$ there is an additional function $\rho$ such that
\begin{equation}
\label{hB2}
\widetilde \mH=\mH+\gamma \ba+\delta \bb+\rho\bc.
\end{equation}
Then in both cases by direct computations, using relations \eqref{EHabs}, we get
$$
\ad_{\tilde\mH}\tilde F\equiv\ad_\mH F +\left((k+w_D-i_D+1)\gamma+3\delta\right)f\mod \hat\Lambda,
$$
where $\tilde F=[E,[\tilde\mH,\ad_E^{k-1}\tilde G]]$. Thus, the lift $\widetilde \mH$ satisfies condition \eqref{cond 4} if and only if
\begin{equation}\label{eq 4}
(k+w_D-i_D+1)\gamma+3\delta=-\xi_0.
\end{equation}
If $i_D>1$ then we have proved in Step 1 that the bundle $B_1$ is reduced to $B_3$ and then $\mH$ is defined uniquely modulo $\bb$. Therefore, if $i_D>1$ then equation \eqref{eq 4} is reduced to $\delta=-\frac{\xi_0}{3}$, which determines the canonical lift of $C$.

If $i_D=1$ then we are looking for one more normalization condition in addition to \eqref{cond 4} in the case $w_D\neq \frac{k+1}{2}$ and two more normalization conditions in the case $w_D=\frac{k+1}{2}$. Let us assume first $w_D=1$. Then $w_D\neq\frac{k+1}{2}$ since $k>1$. Moreover, we can take $G=[\mH,[E,\mH]]$. Then, since $k>1$, $[\mH,G]\in\Gamma(\hat H)$. The distribution $\hat H$ modulo $\hat H_{k+1}$ is spanned by $G,\ad_EG,\ldots,\ad_E^{k-1}G$. If we consider another lift $\tilde\mH$ and the corresponding $\hat G$ then $\ad_E^i\tilde G=\ad_E^iG\mod \hat H_{k+1}$ for any $i$. Therefore the sub-distribution
\begin{equation}
\label{M}
\mathfrak M=\spn\{\ad_E^iG\ |\ i=1,\ldots,k-1\}+\hat H_{k+1}\subset \hat H
\end{equation}
 is well defined.
 % modulo $\hat H_{k+1}$.
%Note that there is no $i=0$ in the definition of $M_{k-1}$ i.e.
We stress that
$G$ is not a section of $\mathfrak M$. Since $G\in\Gamma(\hat K)$, there exists a function $\xi_1$ such that

\begin{equation}
\label{xi1}
\ad_\mH G\equiv\xi_1G\mod \mathfrak M.
\end{equation}
Our additional normalization condition for a lift $\mH$ is
\begin{equation}\label{cond 5}
\ad_\mH G\equiv 0\mod \mathfrak M.
\end{equation}
Clearly $\ad_\mH G=-\ad^3_\mH E$ . If we take  another lift $\widetilde \mH$, then it satisfies \eqref{hB1} or \eqref{hB2} for some functions $\gamma$, $\delta$, and $\rho$. By direct computations we get
$$
\ad_{\tilde \mH}^3E\equiv \ad_\mH^3E-3\left(\gamma+\delta\right)G\mod \widehat H_{k+1}.
$$
Therefore the lift $\widetilde \mH$ satisfies condition \eqref{cond 5} if and only if
\begin{equation}\label{eq 5}
\gamma+\delta=\frac{\xi_1}{3}.
\end{equation}
Equations \eqref{eq 4} and \eqref{eq 5} are independent if and only if $k\neq 2$ (recall that we assume here that $w_D=i_D=1$). However, Corollary \ref{cor 2} says that if $k=2$ then $i_D>1$. In this way conditions \eqref{cond 4} and \eqref{cond 5} determine the canonical lift of $C$ to $B_1$ in the case $w_D=i_D=1$.

If $i_D=1$ and $w_D>1$ then $[\mH,[\mH,E]]\in \hat H_{k+1}$. By Lemma \ref{lemma 2} $w_D\geq 3$.
%Let us assume that $w_D\neq\frac{k+1}{2}$.
Similarly to the previous case of $i_D=1$ we have a sub-distribution
$$\mathfrak N=\spn\{\ad_E^i\mH\ |\ i=1,\ldots,k\}+\widehat V\subset \hat H_{k+1}.$$
%is well defined.
%modulo $\hat V$.
There exists a function $\xi_2$ such that
\begin{equation}
\label{xi2}
\ad^2_\mH E\equiv\xi_2\ad_\mH E\mod \mathfrak N.
\end{equation}
Our additional normalization condition for a lift $\mH$ in this case is
\begin{equation}\label{cond 6}
\ad_\mH^2E\equiv 0\mod \mathfrak N.
\end{equation}
If we take  another lift $\widetilde \mH$ then it satisfies \eqref{hB1} or \eqref{hB2} for some functions $\gamma$, $\delta$, and $\rho$.  By direct computations, using relations \eqref{EHabs}, we get
$$
\ad_{\tilde \mH}^2E \equiv\ad_\mH^2E+2\left(\gamma+\frac{1}{2}\delta\right)\ad_\mH E\mod \widehat V.
$$
Therefore the lift $\widetilde \mH$ satisfies condition \eqref{cond 6} if and only if
\begin{equation}\label{eq 6}
\gamma+\frac{1}{2}\delta=-\frac{\xi_2}{2}.
\end{equation}
Equations \eqref{eq 4} and \eqref{eq 6} are independent if and only if $k+w_D\neq 6$. On the other hand, if $k+w_D=6$ and $w_D>1$ then $k=w_D=3$. However, this situation cannot occur if $[D_{k+1},D_{k+1}]=D$. Indeed, assume that $k=w_D=3$. Let $\ve$ be a section of $V$ and $h$ be a section of $C$. %We can assume that $\ad^4_\ve h=0$.
Then
\begin{equation}
\label{kw31}
[h,\ad_\ve h]=0\mod H_4,\quad [h,\ad^2_\ve h]=0\mod H_4.
\end{equation}
%implies that $[H_4,H_4]\neq H$ which contradicts $[D_4,D_4]=D$.
Applying $\ad_\ve$ to the last relation, we get that
\begin{equation}
\label{kw32}
[\ad_\ve h,\ad_\ve^2 h]+[h,\ad^3_\ve h]\in H_4.
\end{equation}
Applying $\ad_\ve$ to \eqref{kw32} and using the fact that $\ad^4_\ve h \in H_4={\rm span}\{h,\ve, \ad_\ve h, \ad_\ve^2 h, \ad_\ve^3 h\}$ and relations \eqref{kw31}, we get
\begin{equation}
\label{kw33}
[\ad_\ve h,\ad^3_\ve h]\in \mathbb R [h,\ad^3_\ve h]+H_4.
\end{equation}
Finally applying   $\ad_\ve$ to the last relation and using \eqref{kw32} we obtain that
$$[\ad^2_\ve h,\ad^3_\ve h]\in \mathbb R [h,\ad^3_\ve h]+H_4.$$
Thus $\dim [H_4,H_4]/ H_4=1$ and  $[H_4,H_4]\neq H$, which implies that $[D_4,D_4]\neq D$ in contradiction to our genericity assumption (G2).
%However Lemma (**) implies that $k=w_D=3$ contradicts to $[D_{k+1},D_{k+1}]=D$.
So, the case $k=w_D=3$ is impossible.

As a conclusion, in the case when $i_D=1$, $w_D>1$,  and $w_D\neq \frac {k+1}{2}$ conditions \eqref{cond 4} and \eqref{cond 6} determine the canonical lift of $C$ to the bundle $B_1$, while in the case when $i_D=1$ and $w_D=\frac {k+1}{2}$ the same conditions determine a lift of $C$ to the bundle $B_2$ modulo $\mathbb R \bc$.
It remains to kill the freedom in the latter case by introducing one more normalization condition.
For this take a lift $\mH$ of $C$ to $B_2$ satisfying conditions \eqref{cond 4} and \eqref{cond 6}.  One can take
$G=[\mH,(ad_E)^{\frac{k+1}{2}}\mH]$. Since $k\equiv 1\, \mod 4$ and $k>1$, then $k\geq 5$ and therefore $[\mH, [E, G]]$ is a section of $H$. Hence there exists a function $\xi_3$ such that
%There exists a function $\xi_2$ such that
\begin{equation}
\label{xi2}
[\mH,[E, G]]\equiv\xi_3 G \mod \mathfrak M,
\end{equation}
where $\mathfrak M$ is as \eqref{M}
Our last normalization condition for a lift $\mH$ in the considered case is
\begin{equation}\label{cond 7}
[\mH,[E, G]]\equiv 0\mod \mathfrak M.
\end{equation}
If we take  another lift $\widetilde \mH$ satisfying satisfying conditions \eqref{cond 4} and \eqref{cond 6}, then there exists a function $\rho$ such that
\begin{equation}
\label{ hB22}
\widetilde \mH=\mH+\rho\bc
\end{equation}
Let $\widetilde G=[\widetilde \mH,(ad_E)^{\frac{k+1}{2}}\widetilde \mH]$. Then  by direct computations, using relations \eqref{adrel}, we get
$$
 [\widetilde \mH,[E, \widetilde G]]\equiv [\mH,[E, G]]-(k+1)\rho G \widehat H_{k+1}.
$$
Therefore the lift $\widetilde \mH$ satisfies condition \eqref{cond 7} if and only if
%\begin{equation}\label{eq 6}
$\rho=\frac{\xi_3}{k+1}$. Hence, conditions \eqref{cond 4},\eqref{cond 6}, and \eqref{cond 7} fix the lift of $C$ to the bundle $B_2$ uniquely. By this we have completed to lift $C$ to the corresponding bundles $B_i$ in all possible cases.

Finally it is not hard to show that the canonical lift $\mH$ (either to $B_1$ or to  $B_2$) satisfies the following commutative relations:

\begin{equation}
\label{adrela}
[\ba,\mH]=0, \quad [\bb,\mH]=\mH,\quad [\bc,\mH]=0.
\end{equation}
%\end{equation}
To prove these relations one can use arguments similar to those used at the end of step 1 for relations \eqref{adrel}: the distributions $\widehat\Lambda$, $\mathfrak M$, and $\mathfrak N$, appearing in the normalization conditions \eqref{cond 4}, \eqref{cond 5}, \eqref{cond 6}, and \eqref{cond 7}, are invariant with respect to the flow generated by vector fields $\ba$, $\bb$, and $\bc$.

{\bf Step 3. Construction of the canonical frame.}
Now let $E$ and $\mH$ be the canonical lift constructed in the previous steps. We can complete $E$, $\mH$ and the tuple consisting of the fundamental vertical vector fields of the corresponding bundle $B_i$ to the canonical frame $B_i$ by taking appropriate iterative Lie brackets of $E$ and $\mH$.

More precisely, if $i_D\equiv 1$ and $w_D$ is constant and not equal to $\frac{k+1}{2}$ as a canonical frame associated with our distribution on the bundle $B_1$ we can take the tuple of vector fields
\begin{equation}
\label{frame1}
\left(E, \mH, \ad_E \mH, \ldots , \ad_E^k \mH, G,\ldots,\ad_E^{k-1} G, [\mH, \ad_E^{k-1} G], \bigl[E, [\mH, \ad_E^{k-1} G]\bigr],\ba, \bb\right),
\end{equation}
where $G=[\mH, \ad_E^{w_D}\mH]$. If $i_D\equiv 1$ and $w_D\equiv \frac{k+1}{2}$ then as a canonical frame associated with our distribution on the bundle $B_2$ we can take the tuple of the vectors
%we add to the tuple \eqref{frame1} the fundamental vector $\bc$ to obtain the canonical frame on the bundle $B_2$.
\begin{equation}
\label{frame2}
\left(E, \mH, \ad_E \mH, \ldots , \ad_E^k \mH, G,\ldots,\ad_E^{k-1} G, [\mH, \ad_E^{k-1} G], \bigl[E, [\mH, \ad_E^{k-1} G]\bigr],\ba, \bb,\bc\right),
\end{equation}

Further, if $i_D>1$, since $H_{k+1}^2=H$, we can complete the tuple $(E, \mH, \ad_E \mH, \ldots , \ad_E^k\mH, b)$ to the canonical frame on $B_3$ by $k$ vector fields of the type $[\ad_E^s\mH, \ad_E^t\mH]$  for some integer $s$, $t$,  a vector fields of the type $\bigl[\mH, [\ad_E^{\bar s}\mH, \ad_E^{\bar t}\mH\bigr]$ and a vector field of the type $\Bigl[E,\bigl[\mH, [\ad_E^{\bar s}\mH, \ad_E^{\bar t}\mH\bigr]\Bigr]$ for some integers $\bar s$ and $\bar t$. By this we have completed the construction of the canonical frame for corank 2 distributions of the considered in all 3 cases.

Finally, since the fundamental vector fields and the vector field $E$ constitute the frame on each fiber of the bundle $\pi\circ\Pi_i:B_i\mapsto  M$ and these vector fields are the part of the canonical frame, a diffeomorpism of $B_i$, sending the canonical frame of a corank 2 distribution $D$ to the canonical frame of a corank 2 distribution $D'$ (by the pushforward) is  fiberwise. Therefore it induces the diffeomorphism of $M$. The latter diffeomorphism induces the equivalence  between the distributions $D$ and  $D'$, because $(\pi\circ\Pi_i)_*{\rm span}\{\mH, \ad_E \mH, \ldots , \ad_E^k \mH\}=D_{k+1}$ and $D_{k+1}^2=D$.
The proof of Theorem \ref{thm 2} is completed.
%Let us fix $i=1,2,3$ and choose a point $\xi\in B_i$. Let $p=\pi(\xi)\in P(D^\perp)$, where $\pi\colon B_i\to P(D^\perp)$ is the projection. $\xi$ is defined by a vector $h\in C_0(p)$ and additionally by $e_0\in V_0(p)$ or $j^1e_0\in j^1_VV_0(p)$ if $i=2$ or $i=3$. If $i=1$ then $e_0\in V_0(p)$ is uniquely determined by $h$. In any case we can lift $h$ and $e_0$ to the point $\xi$. In this way we obtain two vector fields on $B_i$, denoted $\bh$ and $\be_0$. Recall that $\bb$, $\ba$ and $\tilde{\ba}$ denote fundamental vector fields on $B_i$, which span directions tangent to the fibres of $B_i$. The assumption that $D_{k+1}$ is completely non-holonomic implies that fundamental vector fields together with iterative Lie brackets of $\bh$ and $\be$ generate the whole tangent bundle of $B_i$. In this way we can obtain a frame on $B_i$. We can assume that we use as few as possible Lie brackets of $\bh$ and $\be$ (in this order) and in this way we define the desired canonical frame.
$\Box$

\section{Symmetric models}
\setcounter{equation}{0}
\setcounter{theorem}{0}
\setcounter{lemma}{0}
\setcounter{proposition}{0}
\setcounter{corollary}{0}
\setcounter{remark}{0}

In this section given $k>2$  we find all maximally symmetric models for $(2k+1,2k+3)$-distributions satisfying conditions (G1) and (G2) with respect to the local equivalence. We show that the algebra of infinitesimal symmetries for this models is $(2k+6)$- dimensional if $k\neq 1 \mod 4$ and $(2k+7)$-dimensional if $k\equiv 1\mod 4$, i.e. the upper bounds of Corollary \ref{cor 3} are sharp. By Theorem 2 it may occur only if $i_D \equiv 1$. Note that the case $k=2$ is exceptional, because by Corollary \ref{cor 2} in this case $i_D$ has to be equal to $2$. As was already mentioned in the Introduction, the most symmetric model  for $k=2$ (given by \eqref{sym57})  can be obtained from the analysis of our canonical frame on $B_3$ described in the proof of Theorem \ref{thm 2} but this model can be also recognized without difficulties from the list of $7$-dimensional non-degenerate fundamental graded Lie algebra given in \cite{kuz}, thus we omit this analysis.

%By Theorem \ref{theor 2} these models must satisfy $i_D\equiv 1$
%If $k=2$ then it follows from Corollaty \ref{cor 4} that the dimension of the algebra of infinitesimal symmetries of a %$(5,7)$-distribution $D$ with maximal Kronecker index and completely non-holonomic $D_3$ is bounded from above by 9. If $k>2$ then %Theorem \ref{thm 2} implies that the algebra of infinitesimal symmetries of a $(2k+1,2k+3)$-distribution $D$ with maximal Kronecker %index and completely non-holonomic $D_{k+1}$ is at most $2k+7$ dimensional if  $w_D=\frac{k+1}{2}$ (it can happen only if $k\equiv %1\mod 4$) or at most $2k+6$ dimensional if $w_D\neq\frac{k+1}{2}$. In any case the maximal dimension is attained if and only if all %structural functions of the canonical frame are constant. If it happens, then a distribution is locally equivalent to a left %invariant distribution on a Lie group. We want to find all such distributions.

So, let $k>2$, $i_D\equiv 1$, and $w_D\equiv w$. Then the canonical frame is given by the tuple of vector fields \eqref{frame1} if $w\neq \frac{k+1}{2}$ and by the tuple of vector fields \eqref{frame2} if $w=\frac{k+1}{2}$. For shortness let
\begin{equation*}
\begin{split}
~&\bx_j= \ad_E^j \mH,\quad 0\leq j\leq k,\\
~&\by_j=\ad_E^{j-1} G,\quad 1\leq j\leq k, \\
~&\bz=[\mH, \ad_E^{k-1} G],\quad \bn=\bigl[E, [\mH, \ad_E^{k-1} G].
\end{split}
\end{equation*}
%\qquad \by_1=[\bx_s,\bx_t],\qquad \by_{i+1}=[\be_0,\by_i],
Then in the new notation
\begin{equation}
\label{newnot}
\begin{split}
~& [E,\bx_j]=\bx_{j+1}, \quad0\leq j\leq k-1,\\
~& [\bx_0,\bx_w]=\by_1,\quad [E,\by_j]=\by_{j+1},\,1\leq j\leq k-1,\\
~&[\bx_0,\by_k]=\bz,\quad [E,\bz]=\bn.
\end{split}
\end{equation}

 Denote by ${\rm MS}(k,w)$ the set of all equivalence classes of germs of $(2k+1, 2k+3)$-distributions $D$, satisfying conditions (G1) and  (G2),
relations $i_D\equiv 1$ and $w_D\equiv w$, and  having the algebra of infinitesimal symmetries of the dimension equal to the dimension of the bundle, where their canonical frames are constructed. Take a distribution $D$ representing an element of ${\rm MS}(k,w)$.
  %has the algebra of infinitesimal symmetries of the dimension equal to the dimension of the bundle where its canonical frame is constructed.
This happens if and only if \emph{all structural functions of the canonical frame  of $D$ are constant}. In other words, the vector fields of the canonical frame of $D$ should form the Lie algebra over $\mathbb R$ (that is isomorphic to the algebra of infinitesimal symmetries of the distribution $D$). Denote this algebra by $\mg$.  What properties does this algebra have?
%Let us consider first the case $k>2$. Assume that the algebra of infinitesimal symmetries of a distribution $D$ has maximal possible dimension in the class of all $(2k+1,2k+3)$-distributions with maximal Kronecker index, completely non-holonomic $D_{k+1}$ and fixed value of the invariant $w_D$. It follows that $i_D=1$ (otherwise the algebra of infinitesimal symmetries is at most $2k+5$-dimensional). Then, the vertical coordinates $a$ and $b$ on the canonical bundle define a bi-grading on the canonical frame.
First, combining \eqref{newnot} with \eqref{adrel} and \eqref{adrela} (with $\mH$ replaced by $\bx_0$) and  using the Jacobi identity, one gets
\begin{equation}
\label{abxyz}
\begin{split}
~&[\ba,\bx_j]=j \bx_j,\quad [\ba, \by_j]=(w+j-1)\by_j, \quad [\ba,\bz]=(w+k-1)\bz, \quad[\ba,\bn]=(w+k)\bn,\\
~&[\bb,\bx_j]=\bx_j,\quad [\bb,\quad \by_j]=2\by_j,\quad [\bb, \bz]=3\bz,\quad [\bb,\bn]=3\bn.
\end{split}
\end{equation}
This motivates the introduction of the following natural \emph{bi-grading} on the algebra $\mg$ by assigning to each element of the tuple \eqref{frame1} or \eqref{frame2}
two integer numbers as follows:
\begin{eqnarray}
\label{mcomm}
&~&
\begin{split}
~&\bx_j\mapsto (-j, -1),\,\, \by_j\mapsto (-(w+j-1),-2),\,\,\\
~& \bz\mapsto (-(w+k-1),-3), \,\, \bn\mapsto (-w-k,-3)\\
\end{split}\\
&~&E\mapsto (-1,0), \quad \{\ba,\bb\}\mapsto (0,0), \quad \bc\mapsto (1,0)\label{rest}.
\end{eqnarray}
The above assignment for elements in \eqref{mcomm} is given by the following simple rule: the first integer in the bi-degrees given by \eqref{mcomm} is the number of appearance of $E$ in the representation of the corresponding vector field from the canonical frame as the iterative  brackets of $E$ multiplied by $-1$ and $\mH$ and the second integer there is the number of appearance of $E$ in this representation multiplied by $-1$.
%Indeed, $\be_0$ is homogeneous of degree 1 in $a$, $\bh$ is homogeneous of degree 1 in $b$ and any vector field $\bx$ of a canonical frame is an iterated Lie bracket of $\be_0$ and $\bh$. Therefore it is homogeneous in $a$ and in $b$. It follows that there exists numbers $n_a$ and $n_b$ such that $[\ba, \bx]=n_a\bx$ and $[\bb,\bx]=n_b\bx$ and we say that $\bx$ has degree $(n_a,n_b)$. If $\bx$ has degree $(n_a,n_b)$ and $\by$ has degree $(m_a,m_b)$ then $[\bx,\by]$ has degree $(n_a+m_a,n_b+m_b)$. We immediately get the following:
Let $\mg_{j_1,j_2}$ be the linear span (over $\mathbb R$) of all elements of the canonical frame corresponding to the pair $(j_1, j_2)$.
%is assigned.
Then using relations \eqref{adrel}, \eqref{adrela} and the Jacobi identity,  one gets that
$$[\mg_{j_1,j_2},\mg_{l_1,l_2}]\subset\mg_{j_1+l_1,j_2+l_2},$$
i.e
%\begin{equation}
%\label{bigrad0}
$\mg=\displaystyle{\bigoplus_{(j_1,j_2)\in \mathbb Z^2}}\mg_{j_1,j_2}$
%\end{equation}
 is indeed the bi-grading of the Lie algebra $\mg$.

\begin{definition}
\label{kwtype}
Given $k>2$ and odd $w$ , $1\leq w\leq 2k-1$ , a bi-graded Lie algebra
%\begin{equation}
%\label{bigrad0}
%$\mg=\displaystyle{\bigoplus_{(j_1,j_2)\in \mathbb Z^2}}\mg_{j_1,j_2}$
%\end{equation}
$\mg$ is called a bi-graded Lie algebra of the type $(k,w)$ if the following two conditions hold
\begin{enumerate}
\item
\begin{equation}
\label{mgspan}
\mg =\begin{cases}{\rm span}\{E,\bx_0, \ldots,\bx_k,\by_1,\ldots,\by_k ,\bz,\bn, \ba, \bb\}& {\rm if  }\,\,w\neq \frac{k+1}{2},\\
{\rm span}\{E, \bx_0, \ldots,\bx_k,\by_1,\ldots,\by_k,\bz, \bn,  \ba, \bb,\bc\}& {\rm if  }\,\,w= \frac{k+1}{2}
\end{cases}.
\end{equation}
%, and by the same vectors and the vector $\bc$, if $w\neq \frac{k+1}{2}$
such that   the commutative relations \eqref{newnot}, \eqref{abc}, \eqref{adrel}, and \eqref{adrela} (with $\mH$ replaced by $\bx_0$ in the latter) hold;

\item
the bi-grading on $\mg$ is  given by \eqref{mcomm}-\eqref{rest}.
\end{enumerate}
\end{definition}

So we have shown that \emph{if  the distribution $D$ representing an element of ${\rm MS}(k,w)$
%satisfying conditions (G1) and  (G2),
%relations $i_D\equiv 1$ and $w_D\equiv w$, and has the algebra of infinitesimal symmetries of the dimension equal to the dimension of the bundle where its canonical frame is constructed
then its algebra of infinitesimal symmetries ${\rm symm}(D)$ is a bi-graded Lie algebra of type $(k,w)$}.

Now let
\begin{equation}
\label{frakm}
\mathfrak m=\displaystyle{\bigoplus_{j_2<0}}
%{(j_1,j_2)\in \mathbb Z^2\backslash S}}
\mg_{j_1,j_2}={\rm span}\{\bx_0,\ldots,\bx_k, \by_1,\ldots, \by_k, \bz,\bn\},
\end{equation}
$$\mg'=\displaystyle{\bigoplus_{j_2\geq 0}}
%{(j_1,j_2)\in \mathbb Z^2\backslash S}}
\mg_{j_1,j_2}.$$
Note that
$\mg'={\rm span}\{E,\ba,\bb\}$ if $w_D\neq \frac{k+1}{2}$ and
$\mg'={\rm span}\{E,\ba,\bb,\bc\}$ if $w_D=\frac{k+1}{2}$. Also note that both $\mathfrak m$ is a bi-graded nilpotent subalgebra of $\mg$.
Besides,
\begin{equation}
\label{splitm}
\mg=\mg'\oplus \mathfrak m.
\end{equation}
By the standard arguments the distribution $D$ is locally equivalent to an invariant distribution on the homogeneous space
$G/G'$, where $G$ and $G'$ are
the connected, simply-connected
Lie groups with the Lie algebras $\mg$ and $\mg'$, respectively.
Moreover, from the splitting \eqref{splitm} it follows that \emph{the distribution $D$ is locally equivalent to the left invariant distribution $D_\mg$ on the connected, simply connected Lie group $\mathcal M$ with the Lie algebra $m$ such that}
\begin{equation}
\label{De}
D_\mg(e)={\rm span} \{\bx_0,\ldots,\bx_k,\by_1,\ldots,\by_k\},
\end{equation}
where $e$ is the identity of the group $\mathcal M$. Moreover, we have the following

\begin{proposition}
\label{redbigrade}
The correspondence
%$D\mapsto {\rm symm}(D)$
between the set ${\rm MS}(k,w)$ and the set of all bi-graded Lie algebras of type $(k,w)$, given by  $D\mapsto {\rm symm}(D)$,
%of equivalence classes of germs of $(2k+1, 2k+3)$-distributions $D$, satisfying conditions (G1) and  (G2),
%relations $i_D\equiv 1$ and $w_D\equiv w$, and  having the algebra of infinitesimal symmetries of the dimension equal to the dimension of the bundle, where their canonical frames are constructed, to the set of all bi-graded Lie algebras of type $(k,w)$
is a bijection.
 %The correspondence is given via relation \eqref{De}.
\end{proposition}

\begin{proof}
%By the previous arguments, to prove the proposition it remains to show that
%%\begin{enumerate}
%%\item
%for any bi-graded Li algebra $\mg$ of type $(k,w)$  with the decomposition \eqref{splitm}
%the left-invariant distribution $D_0$ on the Lie group $\mathcal M$ with the Lie algebra $\mathfrak m$, defined by relation %%\eqref{De}, satisfies conditions (G1) and  (G2),
%relations $i_{D_0}\equiv 1$ and $w_{D_0}\equiv w$, and it  has the algebra of infinitesimal symmetries of the dimension equal to the %dimension of the bundle, where their canonical frames are constructed.
%%\end{enumerate}
%To prove this fact  we need the following lemma, which will be also useful in the sequel:
%We have to prove that the above mentioned correspondence between the equivalence classes of distributions from the  considered class having the algebra of infinitesimal symmetries of the dimension equal to the dimension of the bundle, where their canonical frames are constructed, and the set of bi-graded Lie algebras of type $(k,w)$  is a bijection.
First  we prove the following lemma, which will be also useful for other purposes in the sequel:

\begin{lemma}\label{lemma 4}

\begin{itemize}
\item[(A)]
If $\mg$ is a bi-graded Lie algebra of type $(k,w)$ with the basis as in Definition \ref{kwtype}, then
%is spanned by
%$$
%(\be_0,\bx_0,\ldots,\bx_k,\by_1,\ldots,\by_k,\bz,\bn)
%$$
%then
\begin{equation}
\label{xyzn}
\begin{split}
~&[\bx_i,\by_{k-i}]=(-1)^i\bz, \quad 0\leq i\leq k-1\\
~&[\bx_i,\by_{k-i+1}]=(-1)^{i+1}i\bn,\quad 1\leq i\leq k\\
~&[\bx_i,\bx_j]=c_{i,j}\by_{i+j-w+1},
\end{split}
\end{equation}
where $c_{i,j}$ satisfy the following four properties (in addition to the evident antisymmetricity $c_{i,j}=-c_{j,i}$):
\begin{enumerate}
\item
$c_{i,j}=0$ if $i+j<w$ or $i+j>k+w-1$;
%$$
%a_i=i(-1)^{i+1},\qquad b_i=(-1)^i,
%$$
\item $c_{0,w}=1$;
\item
\begin{equation}\label{system 1}
c_{i,j}=c_{i+1,j}+c_{i,j+1};
\end{equation}
\item
\begin{eqnarray}\label{system 2}
&~&(-1)^i c_{0,k+w-1-i}-(-1)^{k+w-1-i}c_{0,i}=c_{i,k+w-1-i}\\
&~&(-1)^{i+1}i c_{0,k+w-i}-(-1)^{k+w-i}(k+w-i+1)c_{0,i}=0.\nonumber
\end{eqnarray}
\end{enumerate}

\item [(B)]
Conversely, if $w\neq \frac{k+1}{2}$ and the tuple $(E,\bx_0,\ldots,\bx_k,\by_1,\ldots,\by_k,\bz,\bn,\ba,\bb) $ satisfy relations \eqref{newnot}, \eqref{abc}, \eqref{adrel}, \eqref{adrela}, and \eqref{xyzn} with  antisymmetric matrix $(c_{i,j})$ satisfy \eqref{system 1} and \eqref{system 2}, then this tuple spans the bi-graded Lie algebra of type $(k,w)$.

\item [(C)]
The matrix $(c_{i,j})$ defines the bi-graded Lie algebra of type $(k,w)$ uniquely, up to an isomorphism.
\end{itemize}
\end{lemma}
\begin{proof}
Throughout this proof we use the fact that
by \eqref{mcomm} and \eqref{rest} the spaces $\mg_{j_1,j_2}$ are at most one-dimensional if $(j_1,j_2)\neq (0,0)$.
Therefore using the compatibility of the Lie brackets with the bi-grading we get that there exists constant $a_i$, $b_i$ and $c_{i,j}$ such that
\begin{equation*}
%\label{xyzn}
\begin{split}
~&[\bx_i,\by_{k-i}]=b_i\bz, \quad 0\leq i\leq k-1\\
~&[\bx_i,\by_{k-i+1}]=a_i\bn,\quad 1\leq i\leq k\\
~&[\bx_i,\bx_j]=c_{i,j}\by_{i+j-w+1}.
\end{split}
\end{equation*}

Constants $a_i$ and $b_i$ can be found by applying the Jacobi identity  to $[E,[\bx_i,\by_{k-i}]]$. On the one hand, we get
$[E,[\bx_i,\by_{k-i}]]=b_i\bn$
and on the other hand
$$
[E,[\bx_i,\by_{k-i}]]=[\bx_{i+1},\by_{k-i}]+[\bx_i,\by_{k-i+1}]=(a_i+a_{i+1})\bn.
$$
Hence we get the equation
$
a_i+a_{i+1}=b_i
$
which holds for any $i=1,\ldots,k-1$. Moreover, if $i=0$ we get
$
a_1=b_0.
$
Similarly we consider $[E,[\bx_i,\by_{k-i-1}]]$ and by Jacobi identity we get the equation
$
b_i+b_{i+1}=0,
$
which holds for any $i=0,\ldots,k-2$. By definition $b_0=1$. In this way we get $b_i=(-1)^i$ and then $a_i=i(-1)^{i+1}$.

In order to get the relation $c_{i,j}=c_{i+1,j}+c_{i,j+1}$ we consider Jacobi identity applied to $[E,[\bx_i,\bx_j]]$, whereas in order to get relations \eqref{system 2} we consider Jacobi identity applied to $[\bx_0,[\bx_i,\bx_{k+w-1-i}]]$ and $[\bx_0,[\bx_i,\bx_{k+w-i}]]$. In this way the part (A) of the Lemma is proved.

Let us prove now the part (B). From the part (A) we know that \eqref{system 1} and \eqref{system 2} are satisfied for any bi-graded Lie algebra of type $(k,w)$. We have to show that there is no other relation on structural constants $a_i$, $b_i$ and $c_{i,j}$. For $w\neq\frac{k+1}{2}$, an additional possibly non-trivial relation can be obtained from Jacobi identity applied to $[\bx_l,[\bx_i,\bx_j]]$, where $l+i+j=k+w-1$ or $l+i+j=k+w$. We will show that all these relations are consequences of \eqref{system 1} and \eqref{system 2}.

We can assume that $l\leq i$ and $l\leq j$. The proof is by induction: we assume that Jacobi identity is satisfied for $[\bx_{l-1},[\bx_{\tilde i},\bx_{\tilde j}]]$ where $l-1\leq\tilde i$ and $l-1\leq\tilde j$. The case $l=0$ corresponds to \eqref{system 2}. If $l>0$ then $\bx_l=[E,\bx_{l-1}]$. Thus
\begin{eqnarray*}
[\bx_l,[\bx_i,\bx_j]]&=& [E,[\bx_{l-1},[\bx_i,\bx_j]]]-[\bx_{l-1},[E,[\bx_i,\bx_j]]]\\
&=&[E,[\bx_{l-1},[\bx_i,\bx_j]]] -[\bx_{l-1},[\bx_{i+1},\bx_j]]-[\bx_{l-1},[\bx_i,\bx_{j+1}]].
\end{eqnarray*}
We used here \eqref{system 1}, which is equivalent to Jacobi identity of brackets involving $E$. By our assumption, we know that Jacobi identity is satisfied by $[\bx_{l-1},[\bx_i,\bx_j]]$, $[\bx_{l-1},[\bx_{i+1},\bx_j]]$ and $[\bx_{l-1},[\bx_i,\bx_{j+1}]]$. Therefore
\begin{eqnarray*}
[\bx_l,[\bx_i,\bx_j]]&=& [E,[[\bx_{l-1},\bx_i],\bx_j]]+[E,[\bx_i,[\bx_{l-1},\bx_j]]]\\
&&-[[\bx_{l-1},\bx_{i+1}],\bx_j]-[\bx_{i+1},[\bx_{l-1},\bx_j]] -[[\bx_{l-1},\bx_i],\bx_{j+1}]-[\bx_i,[\bx_{l-1},\bx_{j+1}]].
\end{eqnarray*}
But, if we use \eqref{system 1} for Lie bracket involving $E$ we get
$$
[E,[[\bx_{l-1},\bx_i],\bx_j]]-[[\bx_{l-1},\bx_{i+1}],\bx_j] -[[\bx_{l-1},\bx_i],\bx_{j+1}]=[[\bx_l,\bx_i],\bx_j]
$$
and
$$
[E,[\bx_i,[\bx_{l-1},\bx_j]]]-[\bx_{i+1},[\bx_{l-1},\bx_j]] -[\bx_i,[\bx_{l-1},\bx_{j+1}]]=[\bx_i,[\bx_l,\bx_j]].
$$
This completes the proof of part (B).

To prove part (C) let us take another basis of the algebra $\mg$ as in Definition \ref{kwtype} and let $\tilde E$, $\tilde \bx_j$ and $\tilde \by_j$ be the corresponding elements of this basis. Then $\tilde E=\alpha E$ and $\tilde \bx_0=\beta \bx_0$ for some $\alpha$ and $\beta$. This together with \eqref{newnot} implies that $\tilde \bx_j=\alpha^j \beta \bx_j$, $\tilde \by_j=\alpha^{w+j-1}\beta^2 \by_j$. Then using the last relation of \eqref{xyzn} we get that $\tilde c_{i,j}=c_{i,j}$, where $\tilde c_{i,j}$ denotes the corresponding constant for the new basis. So, each $c_{i,j}$ is an invariant of the bi-graded Lie algebra of type $(k,w)$, which complete the proof of the last part of the lemma.

\end{proof}

Now fix a bi-graded Li algebra $\mg$ of type $(k,w)$. Let $\mathfrak m$ be as in \eqref{frakm},
%with the decomposition \eqref{splitm}
and let $D_\mg$ be the left-invariant distribution on the Lie group $\mathcal M$ with the Lie algebra $\mathfrak m$, defined by relation \eqref{De}. Then by the first two relations of \eqref{xyzn} it satisfies condition (G1). Further, assume by contradiction that $D_\mg$ does not satisfy
condition (G2). Then from the last relation of \eqref{xyzn} it follows that there exists $l$, $w\leq l\leq k+w-1$ such that $c_{i,j}= 0$ for all $i,j$ such that $i+j=l$. But from relation \eqref{system 1} it follows that  $c_{i,j}= 0$ for all  $i,j$ such that $i+j\leq l$ in contradiction with condition (2) from Lemma \ref{lemma 4}. Conditions (1) and (2) from Lemma \ref{lemma 4} also
imply that $w_{D_{\mg}}=w$.
%So the correspondence between the equivalence classes of distributions from the  considered class having the algebra of infinitesimal symmetries of the dimension equal to the dimension of the bundle, where their canonical frames are constructed, and the set of bi-graded Lie algebras of type $(k,w)$ is onto.
It is also clear by constructions that the group $G$ is a subgroup of the group of symmetries of $D_\mg$.
This together with Corollary \ref{cor 3} implies that the algebra ${\rm symm}(D_\mg)$ of infinitesimal symmetries of $D_\mg$ is isomorphic to $\mg$ as a Lie algebra. Moreover
%according to Tanaka theory
${\rm symm}(D_\mg)$ has natural grading (\cite{tan1}, \cite {yamag}) and
the algebras ${\rm symm}(D_0)$ and $\mg$ are isomorphic as graded Lie algebras, where the grading on $\mg$ is considered with respect to the second bi-degree. Besides, using Definition \ref{kwtype} and relations \eqref{abxyz} it is not hard to show that ${\rm symm}(D_\mg)$ and $\mg$ are isomorphic as bi-graded Lie algebras. It shows that the correspondence $D\mapsto {\rm symm}(D)$ between  the set of all bi-graded Lie algebras of type $(k,w)$, given by  $D\mapsto {\rm symm}(D)$,
%of equivalence classes of germs of $(2k+1, 2k+3)$-distributions $D$, satisfying conditions (G1) and  (G2),
%relations $i_D\equiv 1$ and $w_D\equiv w$, and  having the algebra of infinitesimal symmetries of the dimension equal to the dimension of the bundle, where their canonical frames are constructed, to the set of all bi-graded Lie algebras of type $(k,w)$
is a bijection (with the inverse given by $\mg\mapsto D_{\mg}$). This completes the proof of the proposition.
%\item
%In particular, two non-isomorphic bi-graded Li algebras $\mg$ of type $(k,w)$ produce non-equivalent $(2k+1, 2k+3)$-distributions via %relation  \eqref{De}. So the considered correspondence is an injection, which completes the proof of proposition.
\end{proof}

Now let $\mathfrak L_{k,w}$ denote the set of all bi-graded Lie algebras of type $(k,w)$. Let $\mathfrak L_k=\displaystyle{\bigcup_w} \mathfrak L_{k,w}$,
if $k\nequiv 1 \mod 4$ and $\mathfrak L_k=\mathfrak L_{k,\frac {k+1}{2}}$ if $k\equiv 1 \mod 4$.
From Proposition \ref{redbigrade} it follows that if the set $\mathfrak L_k$ is not empty, then
the problem of finding of all maximally symmetric models of
$(2k+1,2k+3)$-distributions from the considered class is reduced to the description of the set $\mathfrak L_k$.
In the sequel we will do a little bit more, describing the sets $\mathfrak L_{k,w}$ including the case when $k\equiv 1 \mod 4$ but $w\neq \frac{k+1}{2}$. In particular, the set $\mathfrak L_k$ is not empty for any $k>2$ so that Proposition \ref{redbigrade} gives a way to describe the maximally symmetric models. Set

\begin{equation}
\label{dkw}
d(k,w)=\begin{cases}\left[\frac{l-w+1}{3}\right] &\text{ if } k=2l+1\\
\left[\frac{l-w-1}{3}\right]&\text{ if } k=2l
\end{cases}.
\end{equation}

The main result of this section is the following

\begin{theorem}\label{thm 3}
%Assume that $k>2$ and $w\leq\frac{k+1}{2}$ is odd.
The set of bi-graded Lie algebras of type $(k,w)$ is $d(k,w)$-parametric family.
%\begin{enumerate}
%\item
%If $k=2l+1$, then there is $\right[\frac{l-w+1}{3}\left]$-parametric family of bi-graded Lie algebras of type $(k,w)$;
%
%\item
%If $k=2l$, then there is $\left[\frac{l-w-1}{3}\right]$-parametric family of bi-graded Lie algebras of type $(k,w)$;
%
%\item
%If $w>\frac{k+1}{2}$ then the set of bi-graded Lie algebras of type $(k,w)$ is empty.
%\end{enumerate}
\end{theorem}

\begin{remark}
{\rm In particular, if $d(k,w)=0$, then there exists the unique bi-graded Lie algebra of the type $(k,w)$, while if $d(k,w)<0$, then the set of bi-graded Lie algebras of the type $(k,w)$ is empty.} $\Box$
\end{remark}

The proof of this theorem together with Lemma \ref{lemma 4} will give a rather explicit description of all these Lie algebras.
As a direct consequence of Theorem \ref{thm 3} and Proposition \ref{redbigrade} we have the following
\begin{corollary}\label{cor 6}
Let $k>2$

\begin{enumerate}
\item
If $k\equiv 1\mod 4$ and $w=\frac{k+1}{2}$ then there is a unique, up to a local equivalence, $(2k+1,2k+3)$-distribution $D$ satisfying conditions (G1) and (G2), which have $(2k+7)$-dimensional infinitesimal symmetry algebra.
\item
If $w$ is odd and $w\neq \frac{k+1}{2}$,
%is an arbitrary odd number
then
%for any odd $w<\frac{k+1}{2}$
the set of $(2k+1,2k+3)$-distributions $D$ from the considered class that satisfy $w_D=w$ and have $(2k+6)$-dimensional infinitesimal symmetry algebra is a
%$\left[\frac{l-w+1}{3}\right]$
$d(k,w)$-parametric family.
%If $k=2l$,
%%is even
%then for any even $w<\frac{k1}{2} the set of $(2k+1,2k+3)$-distributions $D$ from the considered class which satisfy $w_D=w$ and %$have $2k+6$-dimensional infinitesimal symmetry algebra  is a $\left[\frac{l-w-1}{3}\right]$-parametric family of %$$(2k+1,2k+3)$-distributions from the considered class with $2k+6$-dimensional infinitesimal symmetry algebra.
\end{enumerate}

If $k\nequiv 1 \mod 4$ the families of distributions from the item (2) above are the only distributions from the considered class with $(2k+6)$-dimensional infinitesimal symmetry algebra.
\end{corollary}

{\bf Proof of Theorem \ref {thm 3}.}
By Lemma \ref{lemma 4} in the case $w\neq \frac{k+1}{2}$ the proof of the theorem is reduced to the search  of all antisymmetric matrix $(c_{i,j})$ satisfying conditions (1)-(4) of Lemma \ref{lemma 4}, while in the case $w=\frac{k+1}{2}$ Lemma \ref{lemma 4} guarantees that all bi-graded Lie algebras of type $(k,w)$ are obtained from a subset of such matrices.
%The proof consists of several lemmas
%Our goal is to characterize all antisymmetric matrices $(c_{i,j})$ satisfying conditions of Lemma \ref{lemma 4}.
In a series of lemmas below we bring conditions (3) and (4) of Lemma \ref{lemma 4} in more convenient form.
By relation \eqref{system 1}, all coefficients $c_{i,j}$ are completely determined by $c_{w-1+i,k-i}$ for $i=0,\ldots,k-w+1$. Denote
\begin{equation}
\label{subst1}
x_i={k+w-1\choose i+w-1}c_{i+w-1,k-i}.
\end{equation}
Since $(c_{i,j})$ is antisymmetric, we have
\begin{equation}\label{system 3}
x_i+x_{k-w+1-i}=0, \quad i=0,1,\ldots,\left[\frac{k-w+1}{2}\right].
\end{equation}
%since the matrix $(c_{i,j})$ is anti-symmetric.
\begin{lemma}\label{lemma 5}
Systems \eqref{system 1} and \eqref{system 2} imply
\begin{equation}\label{system 4}
(-1)^i\sum_{j=0}^ix_j=\sum_{j=0}^i{k-j\choose i-j}x_j,
\end{equation}
for $i=0,\ldots,k-w+1$.
\end{lemma}
\begin{proof}
Denote $y_i=c_{i+w-1,k-i}$. We will use \eqref{system 2} and express $c_{0,i}$ in terms of $y_i$. At the beginning $c_{0,k}=y_0$, as follows from the first equation of \eqref{system 2} with $i=w-1$. Then the second equation of \eqref{system 2} gives $c_{0,w}=\frac{w}{k}(-1)^{k+1}y_0$. In the next step we again use the first equation of \eqref{system 2} with $i=w$ and get
$c_{0,k-1}=-y_1-\frac{w}{k}y_0$. Then we proceed by induction and get the formula:
$$
c_{0,k-i}=(-1)^i\sum_{j=0}^i\frac{{k+w-1\choose j+w-1}}{{k+w-1\choose i+w-1}}y_j.
$$
On the other hand it follows from \eqref{system 1} that
$$
c_{0,k-i}=\sum_{j=0}^i{i+w-1\choose j+w-1}y_j.
$$
Now, if we substitute $x_i={k+w-1\choose i+w-1}c_{i+w-1,k-i}$, compare the two expressions for $c_{o,i}$ and use the formula
$$
{i+w-1\choose j+w-1}{k+w-1\choose i+w-1}={k+w-1\choose j+w-1}{k-j\choose i-j}
$$
we get the desired system \eqref{system 4}.
\end{proof}

\begin{remark}
%It follows
\label{rem45}
{\rm
From the proof of Lemmas \ref{lemma 5} it is not hard to see that the space of common solutions of systems \eqref{system 3} and \eqref{system 4} is in one-to one correspondence with the space of antisymmetric matrices $(c_ij)$, satisfying conditions (1), (3), and (4) of Lemma \ref{lemma 4} and the correspondence is given by relation \eqref{subst1}.} $Box$
\end{remark}

Now we analyze the solution space of system \eqref{system 4}.

%in the case $w\neq \frac{k+1}{2}$ a bi-graded Lie algebra of type $(k,w)$ %is completely described by common solution of systems \eqref{system 3} and \eqref{system 4} normalized by condition (2) of Lemma %\ref{lemma 4}, while in the case $w=\frac{k+1}{2}$  a set of  bi-graded Lie algebras of type $(k,w)$ corresponds in general to a %subspace of the space of common solutions of systems \eqref{system 3} and \eqref{system 4} (which, as we will see later, coincides in %fact with the whole space of solutions as in the previous case normalized by condition (2) of Lemma \ref{lemma 4}).
\begin{lemma}\label{lemma 6}
The solution space of \eqref{system 4} is isomorphic to the solution space of the system
\begin{equation}\label{system 6}
\sum_{j=0}^{i-1}{w+i\choose w+j}x_j=0,
\end{equation}
for $i=2,4,6,\ldots,2\left[\frac{k-w+2}{2}\right]$. Moreover, the isomorphism preserves the solution space of \eqref{system 3}.
\end{lemma}
\begin{proof}
If we sum equations corresponding to the indices $i-1$ and $i$ from the system \eqref{system 4} for $i=1,\ldots,k-w+1$, we get the following system of equations
\begin{equation}\label{system 5}
\sum_{j=0}^{i-1}{k+1-j\choose i-j}x_j+\gamma_ix_i=0,
\end{equation}
%for $i=1,\ldots,k-w+1$,
where $\gamma_i=0$ if $i$ is even and $\gamma_i=2$ if $i$ is odd. The first equation from the system \eqref{system 4} with $i=0$ is trivial and we can cross it out.
If $k$ is odd we consider the additional equation with $i=k-w+2$
\begin{equation}\label{system 5a}
\sum_{j=0}^{k-w+1}{k+1-j\choose k-w+1-j}x_j=0.
\end{equation}
We will show later that this equation is a consequence of the other equations from the system \eqref{system 5}.
%At the moment we have a linear system $Ax=0$, where $A$ is $(k-w+2)\times(k-w+2)$ matrix if $k$ is odd or $(k-w+1)\times(k-w+1)$ %matrix if $k$ is even. The matrix $A$ has block lower triangle structure and it is clear that the rank of $A$ is at least %$\frac{k-w+2}{2}$ if $k$ is odd or at least $\frac{k-w+1}{2}$ if $k$ is even.

For any $l=0,\ldots, \left[\frac {k-w}{2}\right]$ the following tuple
\begin{equation}
\label{sol}
\{x_j\}_{j=0}^{2\left[\frac{k-w+2}{2}\right]-1}=\left(\underbrace{0,\ldots, 0}_{2l \text{ times}},\left\{(-1)^j\binom{k-2l+2}{j-2l+1}\right\}_{j=2l}^{2\left[\frac{k-w+2}{2}\right]-1}
%-{k-2l+2\choose 2},\ldots,
%{k-2l+2\choose k-w+1-2l}
%-{k-2l+2\choose k-w+2-2l}
\right)
\end{equation}
is the solution of the system \eqref{system 5} (note that if $k$ is even, then $x_{k-w+1}$ is not involved in system \eqref{system 5}).
%-\eqref{system 5a}.
Indeed, substituting it to this system we get
%the following combinatorial identities:
\begin{equation}
\label{comb1}
\sum_{j=2l-1}^i(-1)^j{k+1-j\choose i-j}\binom{k-2l+2}{j-2l+2}=\sum_{j=2l-1}^i(-1)^j\frac{(k-2l+2)!}{(k-i+1)!(j-2l+1)!(i-j)!}.
\end{equation}
But the right-hand side of the last identity is equal to $0$. To prove this fact express $t^{k-2l+2}=(t-1+1)^{k-2l+2}$ and expand
$(t-1+1)^{k-2l+2}$ into the trinomial expansion. Then the right-hand side of \eqref{comb1} is equal to the coefficient of $t^{k-i+1}$ in this expansion multiplied by $-1$. Therefore it is equal to $0$.

It implies that the rank of the system \eqref{system 5} ( with additional equation \eqref{system 5a} in the case of odd $k$) is at most $\left[\frac{k-w+2}{2}\right]$. On the other hand, from the block lower triangular structure of this system
%\eqref{system 5}-\eqref{system 5a}
it follows
that the equations \eqref{system 5} with even $i$ (together with equation \eqref{system 5a} in the case of odd $k$) are linearly independent. So, the rank of this system
%\eqref{system 5}-\eqref{system 5a}
is equal to $\left[\frac{k-w+2}{2}\right]$ and all equations  of \eqref{system 5} with odd $i$ can be dropped.

%Similarly, if $k$ is even then for any $l=0,\ldots, \frac {k-w-1}{2}$ the tuple \eqref{sol} without the last entry is the solution of the system \eqref{system 5}. Therefore in this case the rank of the system \eqref{system 5}  is at most $\frac{k-w+1}{2}$.
%On the other hand, the equations \eqref{system 5} with even $i$ together with equation \eqref{system 5a} are linearly independent.
%So, the rank of the system \eqref{system 5} is equal to $\frac{k-w+1}{2}$ and all equations  of this system with odd $i$ can be dropped.
%$$
%\begin{array}{ccccccc}
%{k+2\choose 1},&-{k+2\choose 2},&{k+2\choose 3},&-{k+2\choose 4},&{k+2\choose 5},&\ldots&{k+2\choose k-w+2}\\
%0,&0,&{k\choose 1},&-{k\choose 2},&{k\choose 3},&\ldots&{k\choose k-w}\\
%0,&0,&0,&0,&{k-2\choose 1},&\ldots&{k-2\choose k-w-2}\\
%&&&&&\vdots&\\
%0,&0,&0,&0,&0,&\ldots&{w+2\choose 2}
%\end{array}
%$$
%and we check that each of them is a solution of \eqref{system 5} and equation \eqref{system 5a}. Therefore the rank of $A$ is at most $\frac{k-w+2}{2}$. If $k$ is even we take the analogous sequences, but without the last entry, and we prove that the rank of $A$ is at most $\frac{k-w+1}{2}$ in this case.
%
%As a conclusion we get that $A$ has rank $\frac{k-w+2}{2}$ if $k$ is odd or $A$ has rank $\frac{k-w+1}{2}$ if $k$ is even. Therefore, we can omit every second equation from the system \eqref{system 5} (extended by \eqref{system 5a} in the case of odd $k$): we take equations with even $i$. In particular we take equation \eqref{system 5a} if $k$ is odd.

Finally,
%$$
%\begin{array}{ccccccc}
%{k+w+1\choose w},&{k+w+1\choose w+1},&{k+w+1\choose w+2},&{k+w+1\choose w+3},&{k+w+1\choose w+4},&\ldots&{k+w+1\choose k+1}\\
%\end{array}
%$$
the substitution
\begin{equation}
\label{subst}
x_j:={k+w+1\choose w+j}x_j
\end{equation}
for $j=0,\ldots,k-w+1$ transform system \eqref{system 5} to system \eqref{system 6}. Moreover ${k+w+1\choose w+j}={k+w+1\choose w+(k-w+1-j)}$. Hence, the substitution preserves system \eqref{system 3}.
\end{proof}

\begin{lemma}\label{lemma 7}
The solution space of \eqref{system 6} is isomorphic to the solution space of the system
\begin{equation}\label{system 7}
\sum_{j=0}^{2i-1}{\frac{w-1}{2}+i\choose \frac{w+1}{2}-i+j}x_j=0
\end{equation}
for $i=1,2,\ldots,\left[\frac{k-w+2}{2}\right]$, where ${a\choose b}=0$ if $b<0$. Moreover, the isomorphism preserves the solution space of \eqref{system 3}.
\end{lemma}
\begin{proof}
%The solution space of \eqref{system 6} is spanned by the following sequences:
%$$
%\begin{array}{ccccccc}
%{w\choose 0},&-{w+1\choose 0},&{w+2\choose 0},&-{w+3\choose 0},&{w+4\choose 0},&\ldots&{k+1\choose 0}\\
%0,&0,&{w+2\choose 2},&-{w+3\choose 2},&{w+4\choose 2},&\ldots&{k+1\choose 2}\\
%0,&0,&0,&0,&{w+4\choose 4},&\ldots&{k+1\choose 4}\\
%&&&&&\vdots&\\
%\end{array}
%$$
From \eqref{sol} and the substitution \eqref{subst} it follows that the solution space of \eqref{system 6} is spanned by the following tuples
\begin{equation}
\label{sol3}
\left(\underbrace{0,\ldots, 0}_{2l \text{ times}},\left\{(-1)^j\binom{w+j}{w+2l-1}\right\}_{j=2l}^{k-w+1}
%\binom{w+j}{1},
%-{w+j\choose 2},\ldots,
%{k-2l+2\choose k-w+1-2l}
%-{k-2l+2\choose k-w+2-2l}
\right)
\end{equation}
with $l=0,\ldots,\left[\frac {k-w}{2}\right]$.
We claim that \eqref{system 7} has the same solution space. For this first prove the following identity
\begin{equation}
\label{ident2}
\sum_{j\in\mathbb Z}(-1)^j{l\choose j}{s+j\choose m}=
{s\choose m-l},
\end{equation}
where ${a\choose b}=0$ if $b<0$ or $b>a$.
Consider the following polynomial $f(t)=(t+1)^st^l$. On the one hand, the coefficient of $t^m$ in $f$ is equal to  ${s\choose m-l}$. On the other hand,
$$f(t)=(t+1)^s\bigl((t+1)-1\bigr)^l=\sum_{j\in\mathbb Z}(-1)^j{l\choose j}(t+1)^{s+j},$$
so that the coefficient of $t^m$ in $f$ is equal to the left-hand side of \eqref{ident2}. The proof of  identity \eqref{ident2} is completed.

%First we prove the following formula
%\begin{equation}\label{choose}
%\sum_{j\in\mathbb Z}(-1)^j{l\choose t+j}{s+j\choose m}={s-t\choose m-l},
%\end{equation}
%where ${a\choose b}=0$ if $b<0$ or $b>a$. It is sufficient to prove \eqref{choose} for $t=0$ only. We have
%\begin{eqnarray*}
%\sum_{j\geq 0}(-1)^j{l\choose j}{s+j\choose m}&=&
%\sum_{j\geq 0}\sum_{i=0}^j(-1)^j{l\choose j}{j\choose i}{s\choose m-i}\\
%&=&\sum_{j\geq 0}\sum_{i=0}^j(-1)^j{l\choose i}{l-i\choose j-i}{s\choose m-i}\\
%&=&\sum_{i\geq 0}\left(\sum_{j\geq i}(-1)^j{l-i\choose j-i}\right){l\choose i}{s\choose m-i}.
%\end{eqnarray*}
%(We can interchange the summation order since all sums are finite.) But $\sum_{j\in\mathbb Z}(-1)^j{l-i\choose j-i}=0$ for $i\neq l$ and $\sum_{j\in\mathbb Z}(-1)^j{l-i\choose j-i}=1$ for $i=l$. Therefore we get
%
%as required.
Now we substitute a vector \eqref{sol3} from the solution space of \eqref{system 6} to system \eqref{system 7} and use identity \eqref{ident2} with $l=s=\frac{w-1}{2}+i$, $j\mapsto\frac{w+1}{2}-i+j$ and $m=w+2l-1$. We get
\begin{eqnarray*}
\sum_{j=2l}^{2i-1}(-1)^j{\frac{w-1}{2}+i\choose\frac{w+1}{2}-i+j} {w+j\choose w+2l-1}&=&\sum_{j\in\mathbb Z}(-1)^j{\frac{w-1}{2}+i\choose\frac{w+1}{2}-i+j} {w+j\choose w+2l-1}\\
&-&{\frac{w-1}{2}+i\choose\frac{w-1}{2}-i+2l}=0,
\end{eqnarray*}
which proves the lemma.
\end{proof}

\begin{lemma}
\label{ident3lemma}
The following identity holds
\begin{equation}
\label{ident3}
 \sum_{i=0}^{\mu}{\mu\choose i}{\omega+i\choose y-i}= \sum_{i=0}^{\mu}{\mu\choose i}{\omega+i\choose 2\mu+w-y-i}
\end{equation}
\end{lemma}
\begin{proof}
Consider the function $g(t)=(1+t)^\omega(1+t+\frac{1}{t})^\mu$. On the one hand,
\begin{equation*}
%\label{lemma452}
g(t)=(1+t)^\omega\bigl((1+t)+\frac{1}{t})^\mu=\sum_{i=0}^\mu{\mu\choose i}(1+t)^{\omega+i}\frac{1}{t^{\mu-i}}
\end{equation*}
and the coefficient of $t^{y-\mu}$ in the expansion of $g(t)$ into the  powers of $t$ is equal to the left-hand side of \eqref{ident3}. On the other hand,
\begin{equation*}
%\label{lemma452}
g(t)=t^\omega\left(1+\frac{1}{t}\right)^\omega\left(\left(1+\frac{1}{t}\right)+t\right)^\mu=\sum_{i=0}^\mu{\mu\choose i}\left(1+\frac{1}{t}\right)^{\omega+i}t^{\omega+\mu-i}
\end{equation*}
and the coefficient of $t^{y-\mu}$ in the expansion of $g(t)$ into the powers of $t$ is equal to the right-hand side of \eqref{ident3}.
\end{proof}

\begin{proposition}
\label{dimeq}
The solution space of system \eqref{system 3}-\eqref{system 7}
is $(d(k,w)+1)$-dimensional, where $d(k,w)$ is as in \eqref{dkw}.
\end{proposition}
\begin{proof}
First we prove the following
\begin{lemma}\label{lemma 8}
The solution space of system \eqref{system 3}-\eqref{system 7}
is at least $(d(k,w)+1)$-dimensional.
% where $d(k,w)$ is as in \eqref{dkw}.%$\left[\frac{l-w+1}{3}\right]+1$ if $k=2l+1$, and
%$\left[\frac{l-w-1}{3}\right]+1$ if $k=2l$.
%have nontrivial common solution if and only if $w\leq\frac{k+1}{2}$.
\end{lemma}
\begin{proof}
We treat the cases $k=2l+1$ and the case $k=2l$ separately.

{\bf 1) Case $k=2l+1$.}
%First we prove that the dimension of the solution space of system \eqref{system 3}-\eqref{system 7}
%is not less than $\left[\frac{l-w+1}{3}\right]+1$
%if $k=2l+1$, and
%$\left[\frac{l-w-1}{3}\right]+1$ if $k=2l$.
%Consider the case $k=2l+1$.
Fix an integer $s$ such that $0\leq s\leq
%\left[\frac{l-w+1}{3}\right]
d(k,w)$ and let
%Assume $w\leq\frac{k+1}{2}$. Then the number
\begin{equation}
\label{lw}
a_s:=\frac{w+1+4s}{2}, \quad
m_s:=\frac{k-w+2}{2}-s-a_s=\frac{k-2w+1-6s}{2}.
\end{equation}
%belongs to the set $\left\{1,2,\ldots,\left[\frac{k-w+2}{2}\right]\right\}$.
 %If $i\geq l$ then all binomial coefficients from a one row in Pascal triangle in the corresponding equation from system %\eqref{system 7}.
%$\sum_{j=0}^{2i-1}{\frac{w-1}{2}+i\choose \frac{w-1}{2}+1-i+j}x_j=0$ there appear all binomial coefficients from a one row in Pascal triangle.
%Consider all equations from system \eqref{system 7}
%First consider the case when $k$ is odd.
%Let $m=\frac{k-w+2}{2}-l$.
For $i\geq a_s$ multiply the $i$th equation of system \eqref{system 7}  by ${m_s\choose i-a_s}$ and sum up all the obtained equations.
Taking into account \eqref{lw}, we get the following equation
\begin{equation}
\label{lemma451}
\sum_{i=a_s}^{m_s} {m_s+a_s\choose i-a_s} \sum_{j=0}^{2i-1}{\frac{w-1}{2}+i\choose a_s-2s-i+j}x_j=0.
\end{equation}
%(recall that l=\frac{w+1}{2})).
Substituting  $i\mapsto i+a_s$ into \eqref{lemma451} and taking into account \eqref{lw}, we can write \eqref{lemma451} as follows
\begin{equation}\label{sum eq}
\sum_{i=0}^{m_s}{m_s\choose i}\sum_{j=0}^{2i+2a_s-1}{w+2s+i\choose j-2s-i}x_j=0.
\end{equation}

Identity \eqref{ident3}  with $\mu=m_s$, $\omega=w+2s$ and $y=j-2s$ implies that the coefficient of $x_j$ and the coefficient of $x_{2m_s+w+6s-j}$ in \eqref{sum eq} coincide.
%, or, equivalently,
%the following identity holds
%\begin{equation}
%\label{ident3}
% \sum_{i=0}^{\mu}{\mu\choose i}{\omega+i\choose y-i}= \sum_{i=0}^{\mu}{\mu\choose i}{\omega+i\choose 2m+w-y-i}
%\end{equation}
%(we do not assume that $w$ is odd here).
%with $\mu=m_s$, $\omega=w+2s$ and $y=j-2s$.
Note that by \eqref{lw} one has $2m_s+w+6s-j=k-w+1-j$. So, the coefficient of of $x_j$ and the coefficient of $x_{k-w+1-j}$ in \eqref{lemma451} coincide. Thus for any $s$, $0\leq s\leq
 %\left[\frac{l-w+1}{3}\right]
d(k,w)$, the equation of \eqref{system 7}  with $i=\frac{k-w+2}{2}-s$ is a linear combination of other equations from the system  \eqref{system 3}-\eqref{system 7}, which implies the statement of the lemma for odd $k$.
%that the dimension of the solution space of system \eqref{system 3}-\eqref{system 7}
%is not less than $\left[\frac{l-w+1}{3}\right]+1$.
%the equation \eqref{lemma451} is a linear combination of the equations from the system \eqref{system 3}. It implies that the rank of %joined systems \eqref{system 3} and \eqref{system 7} is not maximal and thus it has a nontrivial solution.

{\bf 2) The case of $k=2l$.} This case can be treated similarly. For this fix again an integer $s$ such that $0\leq s\leq %\left[\frac{l-w-1}{3}\right]
d(k,w)$ and let
%Assume $w\leq\frac{k+1}{2}$. Then the number
\begin{equation*}
\label{lw}
a_s:=\frac{w+3+4s}{2}, \quad
m_s:=\frac{k-w+1}{2}-s-a_s=\frac{k-2w-2-6s}{2}.
\end{equation*}
%belongs to the set $\left\{1,2,\ldots,\left[\frac{k-w+2}{2}\right]\right\}$.
 %If $i\geq l$ then all binomial coefficients from a one row in Pascal triangle in the corresponding equation from system %\eqref{system 7}.
%$\sum_{j=0}^{2i-1}{\frac{w-1}{2}+i\choose \frac{w-1}{2}+1-i+j}x_j=0$ there appear all binomial coefficients from a one row in Pascal triangle.
%Consider all equations from system \eqref{system 7}
%First consider the case when $k$ is odd.
%Let $m=\frac{k-w+2}{2}-l$.
then as before for $i\geq a_s$ multiply the $i$th equation of system \eqref{system 7}  by ${m_s\choose i-l_s}$ and sum up all the obtained equations and use identity \eqref{ident3} (with $\mu=m_s$, $\omega=w+1+2s$, and $y=j-2s-1$) to get
%For this let $m=\frac{k-w+1}{2}-l-1$.
%Now assume that $k=2l$.
%For $i\geq l+1$ multiply the $i$th equation of system \eqref{system 7}  by ${m\choose i-l-1}$ and sum up all the obtained equations.
%Then again using the identity \eqref{ident3} (with $w\mapsto w+1$ and $j\mapsto j-1$) we get
that the coefficient of $x_j$ and the coefficient of $x_{k-w+1-j}$ in the considered linear combination of equations from system \eqref{system 7} coincide. Thus for any $s$, $0\leq s\leq
%\left[\frac{l-w-1}{3}\right]
d(k,w)$, the equation of \eqref{system 7}  with $i=\frac{k-w+1}{2}-s$ is a linear combination of other equations from the system  \eqref{system 3}-\eqref{system 7}, which implies the statement of the lemma for an even $k$.
\end{proof}

%Our goal is to prove that \emph {the dimension of the solution space of system \eqref{system 3}-\eqref{system 7}
%is in fact equal to ($d(k,w)$+1)} %is in fact to $\left[\frac{l-w+1}{3}\right]+1$.
By the previous lemma Proposition \ref{dimeq} is equivalent to the fact that the system, obtained from  the system \eqref{system 3}-\eqref{system 7} by crossing out the last $\left[\frac{l-w+1}{3}\right]+1$ equations from the system \eqref{system 7}, has the maximal rank. We call this system the \emph{reduction of system} \eqref{system 3}-\eqref{system 7}. For this let us show that if $x_{k-w+1-2s}=0$ for every $s$ such that
\begin{equation}
\label{ine}
0\leq s\leq %\left[\frac{l-w+1}{3}\right]
d(k,w),
\end{equation}
then the reduction of system \eqref{system 3}-\eqref{system 7} has the trivial solution only.
% We need the following lemma

Indeed, if $x_{k-w+1}=0$ then the first equation of \eqref{system 3} implies that $x_0=0$. Consequently, the first equation of \eqref{system 7} implies that $x_1=0$ and the second equation of \eqref{system 3} implies that $x_{k-w}=0$. In a similar way one can show by induction that from the fact that $x_{k-w+1-2s}=0 $  for all $s$ satisfying \eqref{ine} it follows that $x_j=0$ for  every $j$ such that  $0\leq j\leq 2d(k,w)+1$ or $k-w-2d(k,w)\leq j\leq k-w+1$. But then the remaining variables
$x_j$, $2d(k,w)+2\leq j\leq k-w-2d(k,w)-1$ satisfy the system \eqref{system 3}-\eqref{system 7} with $k$ and $w$  replaced
by $\widetilde k$ and $\widetilde w$,where
\begin{equation}
\label{tilde}
\widetilde k=k-2d(k,w)-2, \quad \widetilde w=w+2d(k,w)+2.
\end{equation}

It is easy to see that
\begin{eqnarray}
&~& \widetilde w>\frac{\widetilde k+1}{2}\quad \text {if  $k$  is odd}\label{tildaineq1}\\
&~& \widetilde w\geq \frac{\widetilde k}{2}\quad \text {if $k$ is even}\label{tildaneq2}
\end{eqnarray}

Not also that if $k$ is even but $k-2w$ is not divided by $6$ then the inequality \eqref{tildaineq1} holds as well.
%or if $k$ is divided by $4$ we have $\tilde w>\frac{\tilde k+1}{2}$.
%So, to finish the proof of
So if $k-2w$ is not divided by $6$ our proposition follows from

\begin{lemma}
\label{w>k}
If $w>\frac{k+1}{2}$, then system \eqref{system 3}-\eqref{system 7} has the trivial solution only.
\end{lemma}
\begin{proof}
By Remark \ref{rem45} it is enough to show that in the considered case the antisymmetric matrix satisfying conditions
(1), (3), and (4) of Lemma \ref{lemma 4} vanishes.

%that the dimension of the solution space of system \eqref{system 3}-\eqref{system 7}
%is not less than $\left[\frac{l-w-1}{3}\right]+1$.
%This completes the case of even $k$.
%\the analogous sum of equations from system \eqref{system 7}, but with indices $i\geq l+1$, and then apply the reasoning from the %case of odd $k$. We take the following combination
%$$
%\sum_{i=l+1}^{\frac{k-w+1}{2}} {\frac{k-w+1}{2}-l\choose i-l} \sum_{j=0}^{2i-1}{\frac{w-1}{2}+i\choose {l-i+j}}x_j=0.
%$$
%and use \eqref{ident3} to prove that the coefficient of $x_j$ and the coefficient of $x_{k-w+1-j}$ there coincide.

%Now, consider all equations from the system \eqref{system 7} with $i=l,l+1,\ldots,\lfloor\frac{k-w+2}{2}\rfloor$. If $k$ is odd take the following combination
%$$
%\sum_{i=l}^{\frac{k-w+2}{2}} {\frac{k-w+2}{2}-l\choose i-l} \sum_{j=0}^{2i-1}{\frac{w-1}{2}+i\choose \frac{w-1}{2}+1-i+j}x_j=0.
%$$
%If $k$ is even take the following combination
%$$
%\sum_{i=l+1}^{\frac{k-w+1}{2}} {\frac{k-w+1}{2}-l\choose i-l} \sum_{j=0}^{2i-1}{\frac{w-1}{2}+i\choose \frac{w-1}{2}+1-i+j}x_j=0.
%$$
%and prove that
%We observe that in the equations above coefficients next to $x_j$ and $x_{k-w+1-j}$ coincide. Therefore the equation is a combination of equations from the system \eqref{system 3}. It follows that the rank of joined systems \eqref{system 3} and \eqref{system 7} is not maximal and thus there is a non-trivial solution.
%Now assume that $w>\frac{k+1}{2}$.
As was already mentioned before,
%the equation $c_{i,j}=c_{i+1,j}+c_{i,j+1}$
condition (3) of Lemma \ref{lemma 4} implies that $c_{i,j}$ are uniquely defined by $c_{i,k+w-1-i}$ for $i=w-1,\ldots,k$ (if $i+j>k+w-1$ then $c_{i,j}=0$). In particular, we can define the mapping
$$
\phi\colon(c_{w-1,k},c_{w,k-1},\ldots,c_{k,w-1})\mapsto (c_{0,k-w+1},c_{1,k-w},\ldots,c_{k-w+1,0}).
$$
 To prove the lemma it is enough to prove that the mapping is bijective.
%definition of a bi-graded Lie algebra of type $(k,w)$,
By condition (1) of Lemma 4.1 $c_{i,j}=0$ for $i+j< w$. If $w>\frac{k+1}{2}$, then $k-w+1<w$. Therefore, if $\phi$ is bijective then $c_{i,k+w-1-i}=0$, and consequently the whole matrix $(c_{i,j})$ vanishes as desired.
%in the contradiction to condition (2) of Lemma \ref{lemma 4}.

The map $\phi$  is the composition of the following two maps
$$
\phi_1\colon(c_{w-1,k},c_{w,k-1},\ldots,c_{k,w-1})\mapsto (c_{0,k},c_{1,k-1},\ldots,c_{k,0})
$$
and
$$
\phi_2\colon(c_{0,k},c_{1,k-1},\ldots,c_{k,0})\mapsto (c_{0,k-w+1},c_{1,k-w},\ldots,c_{k-w+1,0})
$$
defined inductively by recursive relations \eqref{system 1}.
%$c_{i,j}=c_{i+1,j}+c_{i,j+1}$.

It is easy to see that
$$
\phi_1((\underbrace{0,\ldots,0}_i,1,\underbrace{0,\ldots,0}_{k-w+1-i}))=
\left(\underbrace{0,\ldots,0}_i,{w-1\choose 0},{w-1\choose 1},\ldots,{w-1\choose w-1},\underbrace{0,\ldots,0}_{k-w+1-i}\right),
$$
which implies that $\phi_1$ is injective. Besides, $\phi_2$ is surjective, because it is a composition of the maps $$\phi_{2,s}\colon
(c_{0,k-s},c_{1,k-1-s},\ldots,c_{k-s,0})\mapsto (c_{0,k-s-1},c_{1,k-2-s},\ldots,c_{k-s-1,0}), \quad 0\leq s\leq w-2$$ defined by relations \eqref{system 1} and each of this map has a one dimensional kernel and therefore surjective. Moreover, by simple induction
$$
{\rm Ker} \,\phi_2={\rm span} \left\{\left((-1)^j{i+j\choose i}\right)_{j=0}^{k}\right\}_{i=0}^ {w-2}.
$$

Finally  identity \eqref{ident2} (with $l=w-1$, $m=i$ and $s=2i$) implies that spaces $\mathrm{Im}\,\phi_1$ and $\mathrm {Ker}\,\phi_2$ are perpendicular with respect to the standard scalar product in $\R^{k+1}$ (in this case the right-hand side of \eqref{ident2} vanishes because $m-l=i-w+1<0$). Therefore the image of $\phi_1$ is transversal to the kernel of $\phi_2$ .This implies
that $\phi$ is bijective and completes the proof of the lemma.
\end{proof}
%, that will imply the bijectivity of $\phi$.
%Indeed,
%$$
%\phi_1((\underbrace{0,\ldots,0}_i,1,\underbrace{0,\ldots,0}_{k-w+1-i}))=
%\left(\underbrace{0,\ldots,0}_i,{w-1\choose 0},{w-1\choose 1},\ldots,{w-1\choose w-1},\underbrace{0,\ldots,0}_{k-w+1-i}\right)
%$$
%and by a simple induction one can prove that
%It follows that $\mathrm{Im}\,\phi_1\cap{\mathrm Ker}\phi_2=\{0\}$ (e.g. one can check using formula \eqref{ident2} that %$\mathrm{Im}\,\phi_1$ and ${\mathrm Ker}\phi_2$ are perpendicular with respect to the standard scalar product in $\R^{k+1}$).
%The image of $\phi_1$ is spanned by sequences: $(0,\ldots,0,{w-1\choose 0},{w-1\choose 1},\ldots,{w-1\choose w-1},\ldots,0)$. On the other hand, the kernel of $\phi_2$ is spanned by sequences: $$\left({0\choose 0},-{1\choose 0},{2\choose 0},\ldots\right), \left({1\choose 1},-{2\choose 1},{3\choose 1},\ldots\right),\ldots,\left({w-1\choose w-1},-{w\choose w-1},{w+1\choose w-1},\ldots\right).$$ We see that the image of $\phi_1$ is perpendicular to the kernel of $\phi_2$ with respect to the standard product.
It remains to prove the proposition in the case when $k-2w$ is divided by $6$. In this case the corresponding $\widetilde k$ and $\widetilde w$ satisfy $\widetilde w=\frac{\widetilde k}{2}$ and  the proposition will follow from
%\begin{lemma}
%\label{w>k2}
the fact that \emph {if $w$ is odd   and $w=\frac{k}{2}$, then system \eqref{system 3}-\eqref{system 7} has the trivial solution only}.
%\end{lemma}
%\begin{proof}
For this consider the following system of equations
\begin{equation}\label{system 7y}
\sum_{j=0}^{2i-1}{y+i\choose 2i+1-j}x_j=0, \quad i=0,1,\ldots\left[\frac{k-w}{2}\right]
\end{equation}
depending on a parameter $y$ (here ${a\choose b}$ is defined for any $a\in \mathbb C$ and integer $b$ as usual: ${a\choose b}:=\frac{a(a-1)\ldots (a-b+1)}{b!}$ if $b\geq 0$ and
${a\choose b}=0$ if $b<0$).
Note that system \eqref{system 7y} coincides with  system \eqref{system 7} for $y=\frac{ w+1}{2}$. It can be shown that the determinant of the matrix of the system \eqref{system 3}-\eqref{system 7y} is a nonzero polynomial with respect to $y$ such that the set of its roots is the union of the following two sets: the set of all integers between $-\left[\frac{w}{4}\right]$ and $\frac{w-1}{2}$ and the set $\{-\frac{2s-1}{2}:\left[\frac{w}{4}\right]+3\leq s\leq w\}$. In particular, $y=\frac{w+1}{2}$ is not a root of this polynomial, which proves the last statement.
%Assume that the equations of system \eqref{system 3}-\eqref{system 7} are linearly dependent.
%Taking into account that in the considered case the variable $x_{k-w+1}$ does not appear in system \eqref{system 7} and  that the last equation of system \eqref{system 3}
%(i.e. the equation with $i=\frac{k-w+1}{2}=\frac{w+1}{2}$) has the form $x_{\frac{w+1}{2}}=0$, we get easily that the statement of the the lemma is equivalent to the \emph{linear independence of the following set of polynomials}:
%\begin{eqnarray}
%&~& t^i(1+t)^{\frac{w+1}{2}+i},\quad  0\leq i\leq \frac{w-1}{2}, \label{type1}\\
%&~& t^{\frac{w+1}{2}+i}+t^{\frac{3w-1}{2}-i}, \quad  0\leq i\leq
%\frac{w-3}{2}, \label{type2}\\
%&~& t^i,\quad 0\leq i\leq \frac{w-3}{2}.\label{type3}
%\end{eqnarray}
%Since the degree of polynomials from \eqref{type3} does not exceed $\frac{w-3}{2}$ and degrees of polynomials
%from  \eqref{type1} and  \eqref{type2} are at least  $\frac{w+1}{2}$ the last statement is equivalent to the \emph {linear independence of polynomials from  \eqref{type1} and  \eqref{type2}}.
%%??????????????????????????
%The latter statement can be checked straightforwardly.
%In this case $\tilde k$ and $\tilde w$ from \eqref{tilde} satisfy $\tilde w =\frac{\tilde k}{2}$
%\end{proof}

The proof of Proposition \ref{dimeq} is completed.
\end{proof}

To complete the proof of Theorem \ref{thm 3} it remains to prove that the set of solutions of system \eqref{system 3}-\eqref{system 7}, for which the corresponding antisymmetric matrix $(c_{i,j})$ (see Remark \ref{rem45}) satisfies also condition (2) of Lemma \ref{lemma 4}, is an affine subspace of codimension 1 in the solution space of system \eqref{system 3}-\eqref{system 7}. For this let us prove that $x_{k-w+1}\neq 0$ if and only if $c_{0,w}\neq 0$. Indeed, by definition, $x_{k-w+1}\neq 0$ is equivalent to $c_{k,w-1}\neq 0$ and hence to $c_{w-1,k}\neq 0$. Then, it is equivalent to $c_{0,k}\neq 0$ as follows from \eqref{system 1} and finally to $c_{w,0}\neq 0$ as follows from the last equation of \eqref{system 2}. By the same arguments there exist a unique $c\neq 0$ such that $x_{k-w+1}=c$ if and only if $c_{0,w}=1$. Thus $\{x_{k-w+1}=c\}$ is the affine subspace in the solution space of system \eqref{system 3}-\eqref{system 7} we are looking for.
%in such a way that $c_{0,w}=1$.
In this way the theorem is proved in the case $w\neq\frac{k+1}{2}$. It also shows that for $w=\frac{k+1}{2}$ there is at most one bi-graded Lie algebra of type $(k,w)$ (note that $d(k,w)=0$ in this case). On the other hand, from the constructions at the end of the Introduction, using the theory of $\mathfrak {sl}_2$-representation and Proposition \ref{redbigrade} we know that there exists at least one such bi-graded Lie algebra. These proves the theorem in the case $w=\frac{k+1}{2}$ as well.

$\Box$

Corollary \ref{cor 6} implies
that for $k>2$ the unique maximally symmetric model, up to the local equivalence, for the distributions from the considered class exists if and only if $k\equiv 1 \mod 4$ or $k\not \equiv 1 \mod 4$, $d(k,1)=0$, and $d(k,3)<0$. The latter  occurs exactly in the following cases: $k=3$, $4$, $6$. The nontrivial products in the corresponding bi-graded Lie algebras given by \eqref{symm79} in the case $k=3$, by \eqref{symm911} in the case $k=4$, and by \eqref{symm1315} in the case $k=6$  can be directly obtained from conditions (1)-(4) for $c_{i,j}$ listed in  Lemma \ref{lemma 4}. Finally, the set of maximally symmetric models is discrete and consists more than one element if and only if $d(k,1)=0$ and
$d(k,3)=0$. This occurs in the case $k=8$ only and there are exactly two nonequivalent models with $22$-dimensional algebra of infinitesimal symmetries:
one model with $w=1$ and one model with $w=3$. In all other cases the set of maximally symmetric models depend on continuous parameters.
\medskip

{\bf Acknowledgements.}
%\vskip .1in
We would like to thank Boris Doubrov for very useful discussions.
 %especially regarding the most symmetric model for $k\equiv 1\,\mod 4$.
%space of common solutions of systems \eqref{system 3} and \eqref{system 4} is one-dimensional ,which implies that
%\end{proof}

%Theorem \ref{thm 3} and Corollary \ref{cor 5} imply:

%\begin{corollary}\label{cor 6}
%Let $k>2$. If $k=2l+1$, $k\equiv 1\mod 4$ and $w=\frac{k+1}{2}$ then there is a unique $(2k+1,2k+3)$-distribution $D$ which satisfy $w_D=w$ and have $2k+7$-dimensional infinitesimal symmetry algebra. If $k=2l+1$ is an arbitrary odd number then for any odd $w<\frac{k+1}{2}$ there is $\lfloor\frac{l-w+1}{3}\rfloor$-parametric family of $(2k+1,2k+3)$-distributions $D$ which satisfy $w_D=w$ and have $2k+6$-dimensional infinitesimal symmetry algebra. If $k=2l$ is even then then for any odd $w\leq\frac{k+1}{2}$ there is $\lfloor\frac{l-w-1}{3}\rfloor$-parametric family of of $(2k+1,2k+3)$-distributions with $2k+6$-dimensional infinitesimal symmetry algebra.
%\end{corollary}

\end{document}